\documentclass[reqno,11pt]{amsart}
\usepackage[letterpaper,margin=1in,footskip=0.25in]{geometry}
\usepackage[utf8]{inputenc}
\usepackage{mathptmx}
\usepackage[cal=boondoxupr]{mathalfa}
\usepackage{mathrsfs}
\usepackage{setspace, amssymb, amsopn, amsmath, array, pdfsync, url, amsfonts, float,enumitem}
\usepackage[dvipsnames]{xcolor}
\definecolor{chianti}{rgb}{0.6,0,0}
\definecolor{meretale}{rgb}{0,0,.6}
\definecolor{leaf}{rgb}{0,.35,0}

\usepackage[colorlinks=true, hyperindex, citecolor=meretale, urlcolor=leaf, linkcolor=chianti]{hyperref}
\usepackage{tikz, tikz-cd}
\usetikzlibrary{decorations.markings}
\usepackage{fancyvrb,newverbs}
\usepackage[capitalise]{cleveref}
\usepackage{multicol}
\usepackage{booktabs}
\usepackage[all]{xy}

\usepackage[most]{tcolorbox}

\newtcolorbox{commentbox}{
    enhanced,
    breakable,
    colback=white,
    colframe=OliveGreen,
    colbacktitle=OliveGreen,
    coltitle=white,
    fonttitle=\bfseries,
    title=Comment,
}

\newtcolorbox{todobox}{
    enhanced,
    breakable,
    colback=white,
    colframe=Mahogany,
    colbacktitle=Mahogany,
    coltitle=white,
    fonttitle=\bfseries,
    title=To Do,
}

\tikzset{degil/.style={
            decoration={markings,
            mark= at position 0.5 with {
                  \node[transform shape] (tempnode) {$\backslash$};
                  }
              },
              postaction={decorate}
}
}

\definecolor{cverbbg}{gray}{0.93}

\allowdisplaybreaks

\numberwithin{equation}{section}
\newtheorem{Theoremx}{Theorem}
 
\newtheorem{theorem}{Theorem}[section]

\theoremstyle{definition}

\theoremstyle{definition}

\newtheorem{lemma}[theorem]{Lemma}
\newtheorem{notation}[theorem]{Notation}

\newtheorem{proposition}[theorem]{Proposition}
\newtheorem{corollary}[theorem]{Corollary}

\theoremstyle{definition}
\newtheorem{definition}[theorem]{Definition}
\newtheorem{remark}[theorem]{Remark}
\theoremstyle{remark}

\renewcommand{\ker}{\operatorname{ker}}

\newcommand{\Spec}{\operatorname{Spec}}

\usepackage{stmaryrd}

\newcommand{\Soc}{\operatorname{Soc}}
\newcommand{\Hom}{\operatorname{Hom}}

\newcommand{\NN}{\mathbb{N}}

\newcommand{\Proj}{\operatorname{Proj}}
\newcommand{\Span}{\operatorname{Span}}

\newcommand{\Frac}{\operatorname{Frac}}

\newcommand{\Z}{\mathbb{Z}}
\newcommand{\F}{\mathbb{F}}

\newcommand{\PP}{\mathbb{P}}
\newcommand{\A}{\mathbb{A}}

\newcommand{\fm}{\mathfrak{m}}
\newcommand{\fp}{\mathfrak{p}}
\newcommand{\fq}{\mathfrak{q}}
\newcommand{\fa}{\mathfrak{a}}
\newcommand{\fb}{\mathfrak{b}}

\newcommand{\fn}{\mathfrak{n}}

\newcommand{\cO}{\mathcal{O}}

\newcommand{\cP}{\mathcal{P}}

\makeatother

\crefname{Theoremx}{Theorem}{Theorems}
\crefname{setting}{Setting}{Settings}
\crefname{figure}{Figure}{Figures}

\makeatletter
\renewcommand*{\eqref}[1]{%
  \hyperref[{#1}]{\textup{\tagform@{\ref*{#1}}}}%
}
\makeatother

\usepackage[backend=biber, style=numeric, datamodel=mrnumber,sorting=anyt,maxalphanames=5,maxbibnames=99]{biblatex}
\usepackage{hyperref}
\usepackage{xurl}
\hypersetup{breaklinks=true}

\DeclareFieldFormat{mrnumber}{%
  MR\addcolon\space
  \ifhyperref
    {\href{http://www.ams.org/mathscinet-getitem?mr=#1}{\nolinkurl{#1}}}
    {\nolinkurl{#1}}}

\renewbibmacro*{doi+eprint+url}{%
  \iftoggle{bbx:doi}
    {\printfield{doi}}
    {}%
  \newunit\newblock
  \printfield{mrnumber}%
  \newunit\newblock
  \iftoggle{bbx:eprint}
    {\usebibmacro{eprint}}
    {}%
  \newunit\newblock
  \iftoggle{bbx:url}
    {\usebibmacro{url+urldate}}
    {}}

\renewbibmacro{in:}{}

\begin{document}

\title{Bertini's theorem for \texorpdfstring{$F$}{F}-rationality is false}

\author[Polstra]{Thomas Polstra}
\address{Department of Mathematics, University of Alabama, Tuscaloosa, AL 35487 USA}
\email{tmpolstra@ua.edu}
\urladdr{\url{https://thomaspolstra.github.io/}}

\author[Simpson]{Austyn Simpson}
\address{Department of Mathematics, Bates College, 3 Andrews Rd, Lewiston, ME 04240 USA}
\email{asimpson2@bates.edu}
\urladdr{\url{https://austynsimpson.github.io/}}

\begin{abstract}
    Let $k=\overline{\F_2}$. We construct a nine-dimensional $F$-rational affine variety $X$ admitting a locally closed embedding $X\hookrightarrow \PP_k^{19}$ together with a dense open set $U\subseteq (\PP^{19})^\vee$ such that $X\cap H$ is not $F$-rational (or even $F$-injective) for all $H\in U$. We also construct a nine-dimensional projective $F$-rational variety $\mathfrak{X}$ and a closed embedding  $\mathfrak{X}\hookrightarrow\mathbb{P}^{29}_k$ under which the same conclusion holds.
\end{abstract}

\maketitle

\section{Introduction}
This article is concerned with the potential loss of $F$-rationality in a general hyperplane section of a variety, a phenomenon which does not occur for rational singularities over the complex numbers. Consider for motivation the following two persistence conditions for a property $\cP$ of noetherian local rings.

\begin{enumerate}[label=(\Roman*)]
    \item Let $(R,\fm)\to (S,\fn)$ be a flat local ring homomorphism between excellent local rings. If $R$ satisfies $\cP$ and if the closed fiber $S/\fm S$ is a regular local ring, then $S$ satisfies $\cP$.\label{I}
    \item Let $X$ be an algebraic variety over an algebraically closed field $k$, all of whose local rings $\cO_{X,x}$ satisfy $\cP$. If $X\hookrightarrow \PP^n_k$ is a locally closed embedding, then the local rings of general hyperplane section $X\cap H$ also satisfy $\cP$.\label{II}
\end{enumerate}

Over the complex numbers, the property $\cP$ of having rational singularities is known to satisfy both \ref{I} and \ref{II}. The first fact is due to Elkik in her study of deformations of rational singularities (see \cite[Th\'{e}or\`{e}me 5]{Elk78}) while the second is classical: for instance, one can apply Bertini's theorem for smoothness to a resolution and then use the projection formula (see also \cite{Kol97} and \cite[Lemmas 2 and 3]{Bri02}).

In light of the connections between rational and $F$-rational singularities via reduction to prime characteristic \cite{Har98,MS97,Smi97}, there has been an active program for the past few decades to show that $F$-rationality and other Frobenius singularities also satisfy conditions \ref{I} and \ref{II}; see \cite{Abe01,DM24,Ene00,Ene09,PSZ18,QGSS24,SZ13,DSPS25}. For $F$-rationality, \ref{I} is false by \cite{QGSS24} while \ref{II} does hold provided an additional $F$-purity assumption on $X$ \cite{DSPS25}. We show that without this $F$-purity assumption, Bertini's theorem for $F$-rationality can fail.
\begin{Theoremx}\label{mainthm:affine}
Let $k=\overline{\F_2}$. There exists a nine-dimensional $F$-rational domain $R$ finitely generated over $k$, a locally closed embedding $X=\Spec R\hookrightarrow \PP^{19}_k$, and a dense open set $U\subseteq(\PP^{19}_k)^\vee$ such that $X\cap H$ is not $F$-rational for all $H\in U(k)$. In fact, $X\cap H$ is not even $F$-injective.
\end{Theoremx}

To understand the construction of this counterexample, we recall the axiomatic connection established in \cite{CGM86} between the two persistence phenomena above. Consider the closure $Z$ of the incidence correspondence $$\{(x,H)\mid H\in (\PP^n_k)^\vee, x\in H\}\subseteq \PP^n_k\times_k (\PP^n_k)^\vee.$$ Set $Y:=X\times_{\PP^n_k} Z$ and let $\eta$ be the generic point of $(\PP^n_k)^\vee.$ To show that the version of Bertini's theorem described in \ref{II} holds for some property $\cP$, one might try to use \ref{I} to pass the property from $X$ to $Y_\eta\times_{\kappa(\eta)} L$ for finite extensions $\kappa(\eta)\subseteq L$. One can then attempt to spread the property from the geometric generic fiber to general hyperplane sections, thus proving \ref{II}. The examples produced in \cite{QGSS24} were therefore promising evidence that $F$-rationality could fail Bertini's theorem, as they demonstrated the loss of $F$-rationality along a purely inseparable base change $R\to R\otimes_K K^{1/p}$ where $K=\F_p(t)$. The challenge lies in realizing this same pathology in the generic hyperplane section of a variety over an algebraically closed field.

Our construction is built from spreading out a weighted variant of the rings produced in \cite{QGSS24}. More precisely, we construct an $F$-rational ring $R$ for which $\Spec R$ is the total space of a flat family over $\A^6_{\overline{\F}_2}$ whose generic fiber is $F$-rational but not geometrically reduced. We then find a convenient replacement of one of the algebra generators of $R$ yielding an embedding under which Bertini fails. To obtain a projective counterexample, we instead construct a projective compactification of $X$ with explicitly controlled affine charts and a suitable closed immersion into projective space.

\begin{Theoremx}\label{mainthm:projective}
    There exists a nine-dimensional $F$-rational projective variety $\mathfrak{X}$ over $k=\overline{\F_2}$, a closed embedding $\mathfrak{X}\hookrightarrow \PP^{29}_k$, and a dense open set $U\subseteq \left(\PP^{29}_k\right)^{\vee}$ such that for every $H\in U(k)$, $\mathfrak{X}\cap H$ is not $F$-injective.
\end{Theoremx}

The projective counterexample of \cref{mainthm:projective} is less explicit than the affine counterexample of \cref{mainthm:affine}. The nine-dimensional variety $\mathfrak{X}$ is a projective compactification of $X$. Its embedding into $\PP^{29}_k$ is obtained by first constructing an explicit closed immersion $\mathfrak{X}\hookrightarrow\PP^{2547}_k$ and then
restricting the linear projection $\PP^{2547}_k\dashrightarrow\PP^{29}_k$ from a general center $\PP^{2517}_k\subseteq\PP^{2547}_k$.

For comparison, using ideas from \cite{CGM89}, Schwede and Zhang construct a projective surface that is $F$-injective away from finitely many points but whose general hyperplane section is not $F$-injective. This left unanswered whether such a failure can occur for a variety which is either normal or which is $F$-injective \emph{everywhere}; see \cite[Remark 7.6 and Question 8.1]{SZ13}. \cref{mainthm:affine,mainthm:projective} resolve both of these issues because $F$-rationality implies both $F$-injectivity and normality. Finally, the same examples show that Bertini's theorem fails for $F$-anti-nilpotence, a singularity condition lying strictly between $F$-rationality and $F$-injectivity; see for example \cite{EH08,MQ18}.

We prove \cref{mainthm:affine,mainthm:projective} in \cref{section quasi-projective example,sec: projective}, respectively. The proof in the affine setting uses an elementary tight closure calculation, whereas in the projective case we instead use canonical modules and Frobenius duals.

\subsection*{Acknowledgments} 
    We are grateful to Alessandro De Stefani and Wenliang Zhang for helpful discussions. Polstra was supported by a grant from the National Science Foundation DMS \#2502317.
\subsection*{AI Disclosure} These examples were found with the help of ChatGPT-5.6 Sol Pro and ChatGPT-6 Astra Pro. We include an account of how the main theorems were developed.
\begin{enumerate}[label=(\arabic*)]
    \item A prompt intended to narrow the search space was given to ChatGPT-5.6 Sol Pro which included a suggested modification of the hypersurfaces found in \cite{QGSS24}. ChatGPT correctly identified the variety $X=\Spec R$ described in \cref{section quasi-projective example}. However, the locally closed embedding claimed to witness the failure of Bertini's theorem was a certain $X\hookrightarrow \PP^{43}_k$. In addition, there were several inaccuracies in the proof that a general hyperplane section with respect to the suggested embedding was not $F$-injective. In the same response, ChatGPT reported failing to find a \emph{projective} counterexample.
    \item For the same $X=\Spec R$, the authors then identified the simpler embedding $X\hookrightarrow \PP^{19}_k$ presented in \cref{section quasi-projective example} under which a general hyperplane is not $F$-injective.\label{AI:embedding}
    \item With the new embedding in \ref{AI:embedding}, we tasked both ChatGPT-5.6 Sol Pro and ChatGPT-6 Astra Pro (after the latter was released) with finding a projective compactification $\mathfrak{X}$ of $X$ and a \emph{closed} embedding $\mathfrak{X}\hookrightarrow \PP^N_k$ under which a general hyperplane section is not $F$-injective. ChatGPT-6 Astra Pro correctly identified such an embedding with $N=2547$.\label{AI:projective}
    \item The authors identified the simplification of $\mathfrak{X}\hookrightarrow \PP^{29}_k$ via the process outlined in the introduction.
\end{enumerate}
The strategy for proving that general hyperplane sections of these varieties are not $F$-injective has been reworked considerably from the suggestions made by ChatGPT both for simplicity and to reflect the authors' tastes. All writing is due to the authors, although ChatGPT-6 Astra Pro was used in several iterations of proofreading. The authors assume all responsibility for the correctness of this writeup.

\section{Preliminaries}

\subsection{Frobenius and tight closure}
Let $(R,\fm)$ be an $F$-finite $d$-dimensional Cohen--Macaulay local domain of prime characteristic $p>0$ and let $\underline{x}=x_1,\dots, x_d$ denote a system of parameters for $R$. Using the \v{C}ech complex on $\underline{x}$, recall that $R$ is \emph{$F$-injective} if the natural Frobenius map $F:H^d_\fm(R)\to H^d_\fm(R)$ sending $\left[\frac{r}{x_1^n\dots x_d^n}\right]\to \left[\frac{r^p}{x_1^{np}\dots x_d^{np}}\right]$ is injective. The ring $R$ is said to be \emph{$F$-rational} if the tight closure of zero
\begin{align*}
    0^*_{H^d_\fm(R)}=\{\eta\in H^d_\fm(R)\mid \text{there exists } 0\neq c\in R \text{ such that }cF^e(\eta)=0\text{ for all }e\gg0\}
\end{align*}
is the zero submodule of $H^d_\fm(R)$.

Both of these notions may also be restated in terms of Frobenius and tight closure of ideals. Let $I\subseteq R$ be an ideal. Recall that the \emph{Frobenius closure of $I$}, denoted $I^F$, is the ideal consisting of elements $r\in R$ such that $r^q\in I^{[q]}$ for all $q=p^e\gg 0$. The \emph{tight closure of $I$}, denoted $I^*$, comprises the elements $r\in R$ for which there exists $0\neq c\in R$ such that $cr^q\in I^{[q]}$ for all $q\gg 0$. Under our current hypotheses, $R$ is $F$-injective (respectively, $F$-rational) if and only if $\fq=\fq^F$ (respectively, $\fq=\fq^*$) for some ideal $\fq$ generated by a system of parameters; see \cite[Corollary 3.9]{QS17} and \cite[Theorem 4.2(d)]{HH94a}.

Note that if $\fq\subseteq R$ is a parameter ideal, then $R$ is $F$-rational if and only if $\frac{\fq^*}{\fq}$ intersects the socle $\Soc(R/\fq)$ only at zero.

\subsection{Frobenius dual}
\label{subsection: Frobenius dual}

We recall the dual description of Frobenius and its relation to $F$-singularities. See \cite[Section~2.2]{DSPS25} for a more detailed discussion. Let $R$ be an $F$-finite Cohen--Macaulay domain of characteristic $p>0$ with a global canonical module $\omega_R$. Finite duality identifies
\[
    F_*^e\omega_R\cong \Hom_R(F_*^eR,\omega_R).
\]
Under compatible choices of these identifications, evaluation at $F_*^e1$ defines the \emph{Frobenius dual}
\[
    \Phi_R^e:F_*^e\omega_R\longrightarrow\omega_R
\]
obtained as the $\omega_R$-dual of $F^e:R\to F^e_*R$. In particular, if $\Phi_R = \Phi_R^1$, then $\Phi_R^{e+1}=\Phi_R\circ F_*\Phi_R^e$.

These maps commute with localization and completion. For a prime $\mathfrak p\subseteq R$, put $d=\dim R_{\mathfrak p}$. Local duality identifies the completed localization of $\Phi_R^e$
with the Matlis dual of
\[
    F^e:H^d_{\mathfrak pR_{\mathfrak p}}(R_{\mathfrak p})
    \longrightarrow
    F_*^eH^d_{\mathfrak pR_{\mathfrak p}}(R_{\mathfrak p}).
\]
Consequently, $R$ is $F$-injective if and only if $\Phi_R$ is surjective.

For $c\in R$, write
\[
    \Phi_R^e(F_*^ec-):F_*^e\omega_R\longrightarrow\omega_R,
    \qquad
    F_*^e\eta\longmapsto\Phi_R^e(F_*^e(c\eta)).
\]
If $c\neq0$ is a parameter test element, then $R$ is $F$-rational if and only if this map is surjective for some $e\geq1$; see \cite[Proposition~2.3]{DSPS25}.


If $k$ is a perfect field of prime characteristic $p>0$, $S = k[x_1,x_2,\ldots,x_d]$ a polynomial ring over $k$, then $F_*S\cong S^{\oplus p^d} = \Span_S\{F_*x_1^{i_1}x_2^{i_2}\cdots x_d^{i_d}\mid 0\leq i_j<p\}$. If $\Phi_S:F_*S\to S$ is the dual basis element of the basis element $x_1^{p-1}\cdots x_d^{p-1}$, then $\Hom_{S}(F_*S,S) = \Span_{F_*S}\{\Phi_S\}$ serves as the Frobenius dual of the polynomial ring $S$. If $f\in S$ is a nonzerodivisor of $S$, then Fedder implicitly proves $\Phi_S(F_*f^{p-1}-)$ is a lift of the Frobenius dual of the hypersurface $S/(f)$, see \cite[Corollary Page 465]{Fed83} and \cite[Lemma~8.13]{Sta16}. The following proposition, likely known to experts, is the extension of Fedder's observation to $F$-finite Cohen-Macaulay rings with a canonical module and will allow us to explicitly examine the Frobenius dual of a quotient $R/(f)$ when we are able to describe the Frobenius dual of $R$.

\begin{proposition}
    \label{proposition: Fedder for Frobenius dual}
    Let $R$ be an $F$-finite Cohen-Macaulay ring of prime characteristic $p>0$, $\omega_R$ a global canonical module of $R$, and $\Phi_R:F_*\omega_R\to \omega_R$ an identification of the Frobenius dual of $R$. Let $g\in R$ be a nonzero divisor and identify $\omega_{R/(g)}\cong \omega_{R}/g\omega_R$. Then $\Phi_{R}(F_*g^{p-1}-)$ is a lift of the Frobenius dual $\Phi_{R/(g)}$ of $R/(g)$.
\end{proposition}

\begin{proof}
    The claims can be checked after localization at a maximal
    ideal containing $g$ and completion. Thus we may assume that
    $(R,\fm)$ is complete local. Put $d=\dim R$, $S = R/(g)$,
    and identify $\omega_{S}=\omega_R/g\omega_R$.

    There is a commutative diagram with exact rows
    \[
    \xymatrix{
    0\ar[r] & R\ar[r]^{\cdot g}\ar[d]^{F_*g^{p-1}}
      & R\ar[r]\ar[d]^{F} & S\ar[r]\ar[d]^{F} & 0\\
    0\ar[r] & F_*R\ar[r]^{F_*g}
      & F_*R\ar[r] & F_*S\ar[r] & 0.
    }
    \]
    Since $R$ and $S$ are Cohen--Macaulay of dimensions $d$ and $d-1$, respectively, local cohomology gives
    \[
    \xymatrix{
    0\ar[r] & H^{d-1}_{\mathfrak m}(S)\ar[r]\ar[d]^F
      & H^d_{\mathfrak m}(R)\ar[r]^{\cdot g}\ar[d]^{F_*g^{p-1}}
      & H^d_{\mathfrak m}(R)\ar[r]\ar[d]^F & 0\\
    0\ar[r] & H^{d-1}_{\mathfrak m}(F_*S)\ar[r]
      & H^d_{\mathfrak m}(F_*R)\ar[r]^{F_*g}
      & H^d_{\mathfrak m}(F_*R)\ar[r] & 0.
    }
    \]
    Under the compatible local-duality and finite-duality identifications, its Matlis dual is
    \[
    \xymatrix{
    0\ar[r] & F_*\omega_R\ar[r]^{F_*g}\ar[d]^{\Phi_R}
      & F_*\omega_R\ar[r]\ar[d]^{\Phi_R(F_*g^{p-1}-)}
      & F_*\omega_S\ar[r]\ar[d]^{\operatorname{Tr}_{S}} & 0\\
    0\ar[r] & \omega_R\ar[r]^{g}
      & \omega_R\ar[r] & \omega_S\ar[r] & 0.
    }
    \]
    The right-hand square identifies $\operatorname{Tr}_{S}$ with the map induced by $\Phi(F_*g^{p-1}-))$, as claimed. Faithful flatness of completion and localization at all maximal ideals containing $g$ give the assertion for the original ring.
\end{proof}

\section{The quasi-projective example}\label{section quasi-projective example}
Let $k=\overline{\F_2}$ and consider the polynomial rings $A:=k[t_1,\dots, t_6]$ and $B:=A[x_1,x_2,x_3]$. We label the six quadratic monomials in the $x_i$ as 
\begin{align}
    (m_1,\ldots,m_6)=(x_1^2,x_2^2,x_3^2,x_1x_2,x_1x_3,x_2x_3).
\end{align}
Let $f=y^2+\sum\limits_{i=1}^6 t_im_i^2\in B[y]$ and consider the nine-dimensional hypersurface $S=\frac{B[y]}{(f)}.$
The equation $f$ defining $S$ is homogeneous of degree four with respect to the relative grading $\deg(t_a)=0$, $\deg(x_i)=1$, and $\deg(y)=2$. Set $R$ to be the third Veronese subring
\begin{align*}
    R:=S^{(3)}=\bigoplus\limits_{n\geq 0}S_{3n}\subseteq S.
\end{align*}
For each $q=2^e$ define the ideals
\begin{align*}
    \fq:=(t_1,\ldots,t_6,x_1^3,x_2^3,x_3^3)R,\quad \fb_q=(t_1^q,\ldots,t_6^q,x_1^{3q},x_2^{3q},x_3^{3q})B,\quad \fa_{q}:=(t_1^{q},\ldots, t_6^q)A.
\end{align*}
Note also that
\begin{equation}
    S=B\oplus yB \quad \text{ and }\quad\fq^{[q]}S=\fb_q\oplus y\fb_q\label{eq:S-direct-sum}
\end{equation}
 for all $q=2^e$. Let $X=\Spec R$. In what follows, we show that $X$ is $F$-rational and admits a locally closed immersion $X\hookrightarrow \PP^{19}_k$; we then construct a dense open set $U\subseteq (\PP^{19}_k)^{\vee}$ for which $X\cap H$ is not $F$-injective for all $H\in U$.

\begin{notation}
     Let $Z = V(R_+)\subseteq X$ and identify $Z\cong \A_k^6$ with coordinates $t_1,\ldots, t_6$. For a closed point $\tau=(\tau_i)\in Z(k)$, choose values $\lambda_i\in k$ such that $\lambda_i^2=\tau_i$. Set $T_i=t_i+\tau_i$ and let $$z=y+\sum\limits_{i=1}^6\lambda_i m_i\in B[y].$$ Let $\fm_\tau=(T_1,\ldots, T_6)+R_+$ be the maximal ideal of $R$ corresponding to $\tau$.
\end{notation}

\subsection{\texorpdfstring{$X$}{X} is \texorpdfstring{$F$}{F}-rational}
In this subsection, we show that $X=\Spec(R)$ is $F$-rational. The proof proceeds by showing that $X$ is strongly $F$-regular away from the set $Z$. Then we find a tightly closed ideal $\fq\subseteq R$ generated by a system of parameters, which shows that the local ring at the origin is $F$-rational. A degree-preserving involution of $R$ then shows that $X$ is $F$-rational on $Z$.

\begin{lemma}The ring $R$ enjoys the following properties.\label{lem:SFR-away-from-Z}
    \begin{enumerate}[label=(\alph*)]
        \item\label{lem:SFR-away-from-Z algebra generators} The ring $R$ is generated over $A$ by the degree three component $S_3$; that is, by the thirteen elements
        \begin{equation}
            \{x_1^{i_1}x_2^{i_2}x_3^{i_3}\mid i_1+i_2+i_3=3, 0\leq i_j\}\cup \{yx_1,yx_2,yx_3\}.\label{eq:A-algebra-gens}
        \end{equation}
        \item\label{lem:SFR-away-from-Z 9-dimensional} $R$ is a nine-dimensional Cohen--Macaulay normal domain.\label{lem:SFR-away-from-Z-CM}
        \item\label{lem:SFR-away-from-Z SFR away from Z} For all prime ideals $\fp\in X\smallsetminus Z$, the local ring $R_\fp$ is strongly $F$-regular.\label{lem:SFR-away-from-Z-SFR}
    \end{enumerate}
\end{lemma}
\begin{proof}
    Since $S=B\oplus yB$, we see that $R$ is spanned over $A$ by the monomials $x_1^{a_1}x_2^{a_2}x_3^{a_3}$ where $a_1+a_2+a_3\equiv 0\bmod 3$, together with $yx_1^{a_1}x_2^{a_2}x_3^{a_3}$ where $a_1+a_2+a_3\equiv 1\bmod 3$. Monomials of the first type are products of cubic monomials in the $x_i$, and those of the second type are products of some $yx_i$ and cubic monomials in the $x_i$. This proves \ref{lem:SFR-away-from-Z algebra generators}. It follows that $R$ is nine-dimensional and Cohen--Macaulay because $R$ is free over $A[x_1^3,x_2^3,x_3^3]$.

    Observe that $\frac{\partial}{\partial t_1}\left(\sum\limits_{j=1}^6 t_j m_j^2\right)=x_1^4$ so $\sum\limits_{j=1}^6 t_j m_j^2$ is not a square in $\Frac(B)$. It follows that $f$ is irreducible so that $S$ is a domain. Since 
    \begin{align*}
        \frac{\partial}{\partial t_j}f=m_j^2, \quad \frac{\partial}{\partial x_i}f=\frac{\partial}{\partial y}f=0,
    \end{align*}
    we see that the singular locus of $S$ is given by $V_S(x_1,x_2,x_3,y)$. Since this has codimension three, we see that $S$ is normal and hence so too is $R$; this proves \ref{lem:SFR-away-from-Z-CM}. Note that $R_+$ is prime, so $\sqrt{(x_1^3,x_2^3,x_3^3)R}=R_+$ and $X\smallsetminus Z$ is covered by $D_R(x_1^3), D_R(x_2^3)$, and $D_R(x_3^3)$. Finally, each $S[x_i^{-1}]$ is regular and contains $R[x_i^{-3}]$ as a direct summand. It follows (e.g. by \cite[Theorem 3.9]{MP25}) that $R[x_i^{-3}]$ is strongly $F$-regular, as claimed in \ref{lem:SFR-away-from-Z-SFR}.
\end{proof}

\begin{lemma}\label{lem: closed-point-isomorphism}
    For each closed point $\tau\in Z(k)$, we have an isomorphism of local rings $R_{\fm_\tau}\cong R_{\fm_0}$.
\end{lemma}
\begin{proof}
    Note that $$z^2\equiv \sum\limits_{i=1}^6 T_i m_i^2\bmod (f).$$ Then the involution $\Phi_\tau:B[y]\to B[y]$ defined by
    \begin{align*}
        \Phi_{\tau}(t_i)=T_i,\quad \Phi_{\tau}(x_i)=x_i,\quad \Phi_{\tau}(y)=z.
    \end{align*}
    descends to $S$ since $\Phi_\tau(f)=f$. The stated isomorphism then follows from $\Phi_\tau(\fm_0)=\fm_\tau$
\end{proof}
We therefore work primarily with the local ring $R_{\fm_0}$ at the origin of $Z$. We next aim to show that $R_{\fm_0}$ is $F$-rational by showing that the socle $\Soc\left(\frac{R_{\fm_0}}{\fq R_{\fm_0}}\right)$ has trivial intersection with the tight closure $\frac{(\fq R_{\fm_0})^*}{\fq R_{\fm_0}}$.

\begin{definition}
    Let $w:=x_1^2x_2^2x_3^2\in B$ and for each $1\leq i\leq 6$ let $w_i:=\frac{w}{m_i}\in B$.
\end{definition}
\begin{lemma}
    The ideal $\fq$ is a parameter ideal, and the socle $\Soc\left(\frac{R}{\fq}\right)$ is a seven-dimensional $k$-vector space with basis given by the classes of $w,yw_1,\ldots, yw_6.$ In particular, the same property is true of $\Soc\left(\frac{R_{\fm_0}}{\fq R_{\fm_0}}\right)$.
\end{lemma}
\begin{proof}
    Note that $S/\fq S\cong E:=\frac{k[x_1,x_2,x_3,y]}{(x_1^3,x_2^3,x_3^3,y^2)}$, and the projection $S\to R$ identifies $R/\fq$ with the direct sum of the homogeneous elements of $E$ with degrees divisible by three. The quotient $R/\fq$ then admits a monomial $k$-basis given by $$\mathcal{B}=\{y^\ell x_1^{a_1}x_2^{a_2}x_3^{a_3}\mid 0\leq a_i\leq 2, \ell\in\{0,1\}, a_1+a_2+a_3+2\ell\equiv 0\bmod 3\}.$$ In particular, $R/\fq$ is artinian and $\sqrt{\fq}=\fm_0$. It is straightforward to verify that $w$ and $yw_i$ annihilate the maximal ideal of $R/\fq$. Conversely, reasoning by degrees, if $x_1^{a_1}x_2^{a_2}x_3^{a_3}\in\Soc(R/\fq)\cap\mathcal{B}$ then $a_1=a_2=a_3=2$. If $yx_1^{a_1}x_2^{a_2}x_3^{a_3}\in\Soc(R/\fq)\cap\mathcal{B}$ then the bounds force $a_1+a_2+a_3=4$, so the claim follows.
\end{proof}
In proving that $R$ is $F$-rational, we'll need the following identity.
\begin{lemma}\label{lem:congruence}
    For each $1\leq i\leq 6$ and each $q=2^e\geq 2$, we have $(yw_i)^q\equiv t_i^{q/2}w^q\bmod (\fb_q,f)B[y]$.
\end{lemma}
\begin{proof}
Fix $1\leq i\leq 6$. When $j\neq i$, the monomial $m_j w_i$ is divisible by some $x_r^3$. Note then that
\begin{align*}
    (yw_i)^q&\equiv \sum\limits_{j=1}^6 t_j^{q/2}(m_jw_i)^q\bmod fB[y]\\
    &\equiv t_i^{q/2} (m_i w_i)^q\bmod \fb_q\\
    &=t_i^{q/2} w^q. \qedhere
\end{align*}
\end{proof}
\begin{lemma}\label{lem:tight-closure}
    Let $0\neq g\in B$ and let $(\alpha_0,\alpha_1,\ldots, \alpha_6)\in k^7\smallsetminus\{0\}$. Then for all sufficiently large $q=2^e\gg 0$, we have 
\begin{equation}
gw^q\left(\alpha_0^q+\sum\limits_{i=1}^6\alpha_i^q t_i^{q/2}\right)\not\in \fb_q.\label{eq:deg-lemma}
\end{equation}
\end{lemma}
\begin{proof}
    First choose $q=2^e\gg 0$ so that $\deg_{t_i}(g)<\frac{q}{2}$ and $\deg_{x_j}(g)<\frac{q}{2}$ for each $1\leq i\leq 6$ and each $j=1,2,3$. We may therefore write 
    \begin{equation}
        g=\sum\limits_{0\leq a_1,a_2,a_3<q}h_{a_1a_2a_3}x_1^{a_1}x_2^{a_2}x_3^{a_3} \label{eq:deg-lemma-g}
    \end{equation}
    where at least one $h_{a_1a_2a_3}\in A$ is nonzero. Suppose towards a contradiction the left-hand side of \eqref{eq:deg-lemma} was contained in $\fb_q$. Since $\fb_q$ is a monomial ideal, we see that
    $$(\fb_q:w^q)=(t_1^q,\ldots, t_6^q,x_1^q,x_2^q,x_3^q)B$$ so that
    \begin{equation}
g\left(\alpha_0^q+\sum\limits_{i=1}^6\alpha_i^q t_i^{q/2}\right)\in (t_1^q,\ldots, t_6^q,x_1^q,x_2^q,x_3^q)B.\label{eq:deg-lemma-g-2}
    \end{equation}
    Observe that the images of the monomials $$\{x_1^{a_1}x_2^{a_2}x_3^{a_3}\mid 0\leq a_1,a_2,a_3<q\}$$ form a free $A/\fa_q$-basis for $\frac{B}{(t_1^q,\ldots, t_6^q,x_1^q,x_2^q,x_3^q)}$. It then follows from \eqref{eq:deg-lemma-g-2} that 
    \begin{equation}
h_{a_1a_2a_3}\left(\alpha_0^q+\sum\limits_{i=1}^6\alpha_i^q t_i^{q/2}\right)\in\fa_q\label{eq:in-aq}
    \end{equation} for every $0\leq a_1,a_2,a_3<q$. Consider now the $k$-basis of monomials $$\{t_1^{b_1}\cdots t_6^{b_6}\mid 0\leq b_i<q/2\}$$ for $\overline{A}:=\frac{A}{\fa_{q/2}}$. Letting $a_1,a_2,a_3$ be such that $h:=h_{a_1a_2a_3}\neq 0$, it follows from the degree bound $\deg_{t_i}(h)<q/2$ that the image of $h$ in $\overline{A}$ is nonzero. We now divide into cases according to whether $\alpha_0=0$.

When $\alpha_0\neq 0$, \eqref{eq:in-aq} shows that $\alpha_0^q h\in \fa_{q/2}$ which contradicts the previous paragraph. Now suppose $\alpha_0=0$ and choose $i_0$ so that $\alpha_{i_0}\neq 0$. Consider the ideal $$I:=(t_{i_0}^q,t_i^{q/2}\mid i\neq i_0).$$ Note that $\fa_q\subseteq I$ and $$\alpha_0^q+\sum\limits_{i=1}^6\alpha_i^q t_i^{q/2}\equiv \alpha_{i_0}^qt_{i_0}^{q/2}\bmod I$$ so we obtain by \eqref{eq:in-aq} that $h\in (I:_A t_{i_0}^{q/2})=\fa_{q/2}$ which again contradicts $\deg_{t_i}(h)<q/2$.
\end{proof}
    
\begin{theorem}
    The parameter ideals $\fq\subseteq R$ and $\fq R_{\fm_0}$ are tightly closed. Moreover, $R$ is $F$-rational.
\end{theorem}
\begin{proof}
    Let $\eta=\alpha_0w+\sum\limits_{i=1}^6\alpha_i yw_i\in B[y]$ represent a nonzero element of $\Soc(R/\fq)$, and use the same notation for its image in $R$. Suppose that $\eta\in\fq^*\subseteq R$. Then there exists $0\neq c\in R$ such that $c\eta^q\in\fq^{[q]}$ for all $q=2^e\gg 0$. Write $c=c_0+yc_1$ where at least one $c_j\in B$ is nonzero. By \cref{lem:congruence} we see
    \begin{equation}
        \eta^q\equiv w^q\left(\alpha_0^q+\sum\limits_{i=1}^6\alpha_i^q t_i^{q/2}\right)\bmod (\fb_q,f)B[y]\label{eq:eqn-q-power-sum}
    \end{equation}
    so in particular both sides of \eqref{eq:eqn-q-power-sum} may be viewed as elements of $B$. From \eqref{eq:S-direct-sum} it then follows that for $j=0,1$ we have
    \begin{equation*}
c_jw^q\left(\alpha_0^q+\sum\limits_{i=1}^6\alpha_i^qt_i^{q/2}\right)\in \fb_q.
    \end{equation*} This contradicts \cref{lem:tight-closure} and since $R/\fq$ is artinian we conclude that $\fq^*=\fq$. Since tight closure commutes with localization for parameter ideals (see for example \cite[Proposition 4.14]{HH90}), it follows that $(\fq R_{\fm_0})^*=(\fq R_{\fm_0})$ so that $R_{\fm_0}$ is $F$-rational. Then $R_{\fm_\tau}$ is $F$-rational for all closed points $\tau\in Z(k)$ by \cref{lem: closed-point-isomorphism}. Since $F$-rationality localizes, $R_\fp$ is $F$-rational whenever $\fp\subseteq \fm_\tau$ for some $\tau\in Z(k)$. Finally, $R_\fp$ is strongly $F$-regular for all $\fp\in X\smallsetminus Z$ by \cref{lem:SFR-away-from-Z} so that $R$ is $F$-rational.
\end{proof}

\subsection{\texorpdfstring{$X\cap H$}{X∩H} is not \texorpdfstring{$F$}{F}-injective for general \texorpdfstring{$H\subseteq \PP^{19}_k$}{H⊆ℙ¹⁹ₖ}}

We first modify the generating set for $R$ to produce a locally closed embedding in $\PP^{19}_k$ that is a counterexample to Bertini's Theorem for $F$-rationality. Label the generators in \eqref{eq:A-algebra-gens} for $R$ as an $A$-algebra as $g_1,\ldots, g_{13}$, distinguishing $g_1=x_1x_2x_3$ and $g_{11}=yx_1$, $g_{12}=yx_2$, $g_{13}=yx_3$. Considering $R$ as a $k$-algebra, we then replace $t_6$ with the generator
\begin{equation}
    \theta:=t_6+t_1^2x_1x_2x_3+t_2^2 yx_1+t_3^2 yx_2+t_4^2 yx_3.
\end{equation}
Take homogeneous coordinates $[Z_0:\cdots :Z_{19}]$ for $\PP^{19}_k$ and consider the embedding
\begin{align*}
    \iota_\theta:&X\hookrightarrow \PP^{19}_k\\
    &x\mapsto [1:t_1(x):\cdots:t_5(x):\theta(x):g_1(x):\cdots:g_{13}(x)].
\end{align*}
Let $[a_0:\cdots:a_{19}]$ be the corresponding homogeneous coordinates on the dual space $(\PP^{19}_k)^\vee$ so that $[a]$ parametrizes the hyperplane $H_a=V(a_0Z_0+\cdots+a_{19}Z_{19})\subseteq \PP^{19}_k$. Now define the open set
\begin{align*}
    U:=\{[a_0:\cdots:a_{19}]\in (\PP^{19}_k)^\vee \mid a_6\neq 0\}.
\end{align*}
We will show that $\iota_\theta^{-1}(H_a)$ is not $F$-injective for every $[a]\in U$. 
\begin{theorem}
    For every hyperplane $H\in U(k)$, the hyperplane section $\iota_\theta^{-1}(H)=X\cap H$ is not $F$-injective.
\end{theorem}
\begin{proof}
    We show in fact that for every $H\in U(k)$, there exists an affine line $\Gamma_H\subseteq Z\cap (X\cap H)$ such that the local ring $\cO_{X\cap H,x}$ is not $F$-injective for every closed point $x\in \Gamma_H$. Let $[a]\in U$ and write $H_a=V\left(\sum\limits_{j=0}^{19} a_j Z_j\right)$. Let $b_i=\frac{a_i}{a_6}$ for $0\leq i\leq 6$. One checks that restricting $H_a$ to $X$ under $\iota_\theta$ gives the element
    \begin{align}
        L=h_1+h_2+(t_1^2+\alpha_0)x_1x_2x_3+\sum\limits_{i=1}^3 (t_{i+1}^2+\alpha_i)yx_i\label{eq:Line-not-F-injective}
    \end{align}
    where
    \begin{align*}
        h_1:=b_0+\sum\limits_{j=1}^6 b_j t_j,\quad h_2:=\sum\limits_{j=2}^{10}d_j g_j
    \end{align*}
    for some $d_j\in k$. Let $r_{i+1}=\sqrt{\alpha_i}$ for $0\leq i\leq 3$ and define the affine line
    \begin{align*}
        \Gamma_{H_a}:=V_Z(t_1-r_1,t_2-r_2,t_3-r_3,t_4-r_4,h_1).
    \end{align*}
Fix a closed point $\tau\in \Gamma_{H_a}(k)$. Let $\lambda_j=\sqrt{\tau_j}$ and consider the change of coordinates
\begin{align*}
    T_j=t_j-\tau_j,\quad Y=y+\sum\limits_{j=1}^6\lambda_j m_j.
\end{align*}
Note that $Y^2=\sum\limits_{j=1}^6 T_j m_j^2$ and \eqref{eq:Line-not-F-injective} may be rewritten as
\begin{align}
L=\widetilde{h}_1+h_2+T_1^2x_1x_2x_3+\sum\limits_{i=1}^3 T_{i+1}^2\left(Y+\sum\limits_{j=1}^6 \lambda_j m_j\right)x_i\label{eq:translated-L}
\end{align}
where $\widetilde{h}_1=\sum\limits_{j=1}b_j T_j$. Now let $\fn=(T_1,\ldots, T_6)R+R_+\subseteq R$ and denote $D:=R_\fn/LR_\fn$. Note that the degree zero term $\widetilde{h}_1$ of $L$ is nonzero, so $D$ is an eight-dimensional Cohen--Macaulay local ring by \cref{lem:SFR-away-from-Z}\ref{lem:SFR-away-from-Z-CM}. Consider the ideals
\begin{align*}
    J&:=(T_1,\dots, T_5,x_1^3,x_2^3,x_3^3)D\\
    I&:=(T_1,\dots, T_5,x_1^3,x_2^3,x_3^3,L)R
\end{align*}
and note that $R/I$ is supported at $\fn/I$ so that $D/J\cong (R/I)_{\fn/I}$. We aim to show that $J$ is a parameter ideal of $D$. By pigeonhole we have $((x_1,x_2,x_3)B)^7\subseteq (x_1^3,x_2^3,x_3^3)B$, so in particular $h_2 m_6^2\in (x_1^3,x_2^3,x_3^3)B$. It follows that
\begin{align*}
    S/IS &\cong \frac{k[T_6,x_1,x_2,x_3,Y]}{(T_6+h_2,x_1^3,x_2^3,x_3^3,Y^2+T_6m_6^2)}\\
    &\cong \frac{k[x_1,x_2,x_3,Y]}{(x_1^3,x_2^3,x_3^3,Y^2+h_2 m_6^2)}\\
    &\cong \frac{k[x_1,x_2,x_3,Y]}{(x_1^3,x_2^3,x_3^3,Y^2)}=:E.
\end{align*}
Since the projection $S\to R$ is $R$-linear and restricts to the identity on $R$, we have $IS\cap R=I$ and an induced injection $R/I\hookrightarrow E$ where $E$ is a finite dimensional $k$-vector space. It follows that $R/I$ and its localization $D/J$ are both artinian so that $J$ is a parameter ideal of $D$.

Let $\beta_j=\sqrt{b_j}$ and denote $\xi=Y\sum\limits_{j}^6\beta_j w_j\in D$. Note that each $Y w_j$ has degree six, so $\xi\in R$. We claim that $\xi\in J^F\smallsetminus J$. To see that $\xi\not\in J$, simply note that the term $Yx_1^2x_2x_3$ appears with nonzero coefficient. We next show that $\xi^2\in J^{[2]}$. In the hypersurface $S$, we have the identity
\begin{align}
    \xi^2&=Y^2\sum\limits_{j=1}^6 b_j w_j^2\nonumber\\
    &=\sum\limits_{i=1}^6\sum\limits_{j=1}^6 T_i b_j (m_i w_j)^2\nonumber\\
    &=\widetilde{h}_1w^2+\sum\limits_{i\neq j}^6 T_i b_j (m_i w_j)^2.\label{eq:xi-h1}
\end{align}
For $i\neq j$ write $m_i=x_1^{a_1}x_2^{a_2}x_3^{a_3}$ and $m_j=x_1^{e_1}x_2^{e_2}x_3^{e_3}$ where $a_1+a_2+a_3=e_1+e_2+e_3=2$. Choose an index $r$ for which $a_r>e_r$ and write in $R$
\begin{align}
    m_iw_j&=\prod\limits_{\ell=1}^3 x_\ell^{2+a_\ell-e_\ell}\nonumber\\
    &=x_r^3\frac{m_iw_j}{x_r^3}.\label{eq:monomial-factor-in-R}
\end{align}
Note that the expression in \eqref{eq:monomial-factor-in-R} is valid in $R$ because the second factor is cubic in the $x_i$. It follows that
\begin{equation}
    \sum\limits_{i\neq j}^6 T_i b_j (m_i w_j)^2\in (x_1^6,x_2^6,x_3^6)R.\label{eq:i-not-j}
\end{equation}
Using \eqref{eq:translated-L} we also observe (in the ring $D$) the identity 
\begin{align*}
    \widetilde{h}_1w^2=h_2w^2+T_1^2 g_1 w^2+\sum\limits_{i=1}^3 T_{i+1}^2\left(Y+\sum\limits_{j=1}^6 \lambda_j m_j\right)x_i w^2.
\end{align*}
One verifies at once that $h_2 w^2\in (x_1^6,x_2^6,x_3^6)R$, so that by \eqref{eq:xi-h1} and \eqref{eq:i-not-j} we have
\begin{align*}
    \xi^2\equiv T_1^2 g_1w^2+\sum\limits_{i=1}^3 T_{i+1}^2\left(Y+\sum\limits_{j=1}^6 \lambda_j m_j\right)x_i w^2 \bmod (x_1^6,x_2^6,x_3^6)D.
\end{align*}
It follows that
\begin{align*}
    \xi^2\in (T_1^2,T_2^2,T_3^2,T_4^2,x_1^6,x_2^6,x_3^6)D\subseteq J^{[2]}
\end{align*}
as claimed, hence $D$ is not $F$-injective by \cite[Corollary 3.9]{QS17}.
\end{proof}

\section{The Projective Example}\label{sec: projective}

\subsection{\texorpdfstring{$F$}{F}-rational charts}\label{section affine charts}

To simplify the indexing in what follows, consider the seven-element set $$\Lambda:=\{c=(c_1,c_2,c_3,d)\in\Z_{\geq 0}^4\mid c_1+c_2+c_3+2d=2\}.$$ Let $k=\overline{\F_2}$ and for each $a\in \Lambda$ consider the polynomial rings $A_a:=k[t_b\mid b\in \Lambda\setminus\{a\}]$ and $B_a:=A_a[x_1,x_2,x_3,y]$. For each $a=(a_1,a_2,a_3,b)\in \Lambda$ let 
\[
m_a = x_1^{a_1}x_2^{a_2}x_3^{a_3}y^b,\quad f_a = m_a^2 + \sum_{b\in\Lambda\setminus\{a\}}t_bm_b^2,\quad \mbox{and } C_a = B_a/(f_a).
\]
We assign the variables $x_1,x_2,x_3$ degree $1$, $y$ degree $2$, and for each $a\in\Lambda$, $t_a$ the degree $0$. For each $n\in\NN$ let $\mathcal{E_n} = \{x_1^{\alpha_1}x_2^{\alpha_2}x_3^{\alpha_3}y^{\beta}\mid \alpha_1+\alpha_2+\alpha_3+2\beta = n\}$, the collection of degree $n$ monomials in the variables $x_1,x_2,x_3,y$. The set $\Lambda$ labels the exponent vectors of the monomials in $\mathcal{E}_2$ and
\[
\{m_b\mid b\in\Lambda\} = \{y,x_1^2,x_1x_2,x_1x_3,x_2^2,x_2x_3,x_3^2\} = \mathcal{E}_2
\]
consists of seven quadratic monomials in $B_a$ for every $a\in\Lambda$. Moreover, $f_a$ is homogeneous of degree $4$ and for each $a\in\Lambda$ we let $R_a = (C_a)^{(3)}$, the third Veronese of $C_a$. If $a = (0,0,0,1)$, then $X=\Spec(R_{(0,0,0,1)})$ is the affine variety of \cref{section quasi-projective example} that admits a locally closed embedding into $\PP^{19}_k$ so that if $H$ is a general hyperplane of $\PP^{19}_k$, then $X\cap H$ is not $F$-injective.

\begin{lemma}
    \label{lem: F rationality of Veronese general}
    For each $a\in \Lambda$, the ring $R_a$ enjoys the following properties:
    \begin{enumerate}[label=(\alph*)]
        \item\label{lem: F rationality of Veronese general group action} Let $\zeta\in k$ be a primitive third root of unity and $\mu_3=\langle\zeta\rangle = \{ 1,\zeta,\zeta^2\}$. Then $R_a$ is the invariant ring of a $\mu_3$-action on $C_a$ with the properties $\zeta\times x_i\mapsto \zeta x_i$, $\zeta\times y\mapsto \zeta^2y$, and $\zeta\times t_b = t_b$ for all $b\in\Lambda\setminus\{a\}$.
        \item\label{lem: F rationality of Veronese general CM and normal} $R_a$ is a nine-dimensional Cohen--Macaulay normal domain.
        \item\label{lem: F rationality of Veronese general SFR locus} For every $\fp\in \Spec(R_a)\smallsetminus V((R_a)_+)$, the scheme $\Spec(R_a)$ is smooth over $k$ at $\fp$.
        \item\label{lem: F rationality of Veronese general canonical module} $\omega_{R_a} \cong \oplus_{n\in\NN} (C_a)_{3n-1}$ is a canonical module of $R_a$. 
        \item\label{lem: F rationality of Veronese general restriction of Frobenius dual} Let $\Phi_{B_a}:F_*B_a\to B_a$ denote the Frobenius dual of $B_a$ with respect to the $p$-basis 
        \[
        \{t_b\mid b\in\Lambda\setminus\{a\}\}\bigcup\{x_1,x_2,x_3,y\}
        \]
        Let $\Phi_{C_a}$ be the Frobenius dual of $C_a$ obtained as the restriction of $\Phi_{B_a}(F_*f_a-)$ to the quotient $C_a$. Identify $\oplus_{n\in\NN} (C_a)_{3n-1}$ as the canonical module $\omega_{R_a}$ of $R_a$.
        \begin{itemize}
            \item For every $n\in\NN$, 
                \[
                \Phi_{C_a}((F_*\omega_{R_a})_n)\subseteq (\omega_{R_a})_n.
                \]
            \item The restriction of the domain of the Frobenius dual map $\Phi_{C_a}:F_*C_a\to C_a$ to $\Phi_{C_a}:F_*\omega_{R_a}\to \omega_{R_a}$ is the Frobenius dual of $R_a$.
        \end{itemize}
        \item\label{lem: F rationality of Veronese general F-rational} $R_a$ is $F$-rational.
    \end{enumerate}
\end{lemma}
\begin{proof}
    Consider the $\mu_3$-action on $B_a$ defined by $\zeta\times x_i\mapsto \zeta x_i$, $\zeta\times y\mapsto \zeta^2y$, and $\zeta\times t_b = t_b$ for all $b\in\Lambda\setminus\{a\}$. Then $\zeta\times f_a = \zeta^4f_a$ and there is an induced action on $C_a = B_a/(f_a)$ as claimed. The action of $\zeta$ on $(C_a)_d$ is multiplication by $\zeta^d$. Therefore $R_a$ is the invariant ring of the $\mu_3$-group action on $C_a$. The hypersurface $C_a$ is nine-dimensional and Cohen-Macaulay. The map $R_a\to C_a$ is finite, therefore $R_a$ is nine-dimensional. By \cite[Proposition~13]{HE71}, $R_a$ is Cohen-Macaulay. 
    
    To show $R_a$ is normal it suffices to show $C_a$ is normal since invariant rings of a finite group action on a normal domain are normal. For each $b\in\Lambda\setminus\{a\}$, 
    \[
    \frac{\partial f_a}{\partial t_b} = m_b^2
    \]
    and $B_a/((f_a)+ (m_b^2\mid b\in\Lambda\setminus\{a\}))$ is six-dimensional. The height three ideal $\left(m_b^2\mid b\in\Lambda\setminus\{a\}\right)C_a$ defines the singular locus of $\Spec(C_a)$. Therefore $C_a$ is a normal domain by Serre's criterion for normality.

    We next verify if $\fp\in \Spec(R_a)\setminus V((R_a)_+)$ then $\Spec(R_a)$ is smooth at $\fp$ over $k$. Equivalently, $(R_a)_{\fp}$ is a regular local ring as $k$ is perfect. For such a prime $\fp$, either $x_i^3\not\in \fp$ for some $1\leq i\leq 3$ or $y^3\not\in \fp$. If $1\leq i\leq 3$, $c\in C_a$ is homogeneous, then there exists $0\leq r<3$ so that $c = (cx_i^{-r})x_i^r$ and $cx_i^{-r}$ has degree divisible by three. Therefore $C_a[x_i^{-1}]\cong R_a[x_i^{-3}][U]/(U^3-x_i^{-3})$. Similarly, for homogeneous $c\in C_a$ there exists $0\leq r<3$ so that $c = (cy^{-r})y^r$, $cy^{-r}\in R_a[y^{-3}]$, and $C_a[y^{-1}]\cong R_a[y^{-3}][U]/(U^3-y^{-3})$. Therefore the finite maps $R_a[x_i^{-3}]\subseteq C_a[x_i^{-1}]$ and $R_a[y^{-3}]\subseteq C_a[y^{-1}]$ are faithfully flat and \'{e}tale. The algebras $C_a[x_i^{-1}]$ and $C_a[y^{-1}]$ are regular by the Jacobian criterion, hence $R_{\fp}$ is regular for every $\fp\in \Spec(R_a)\setminus V((R_a)_+)$.

    As an $A_a$-algebra, $B_a$ is generated by the four positive degree elements $x_1,x_2,x_3,y$ of degrees $1,1,1,2$ respectively. Therefore $\omega_{B_a} = B_a(-(1+1+1+2))=B_a(-5)$, see \cite[\S 3.6]{BrunsHerzog}. The element $f_a\in B_a$ is homogeneous of degree $4$, hence $\omega_{C_a} = C_a(4-5) = C_a(-1)$. By \cite[Corollary~3.1.3]{GW78}, $\omega_{R_a} \cong (\omega_{C_a})^{(3)}$. In particular, $(\omega_{R_a})_n = C_a(-1)_{3n} = (C_a)_{3n-1}$.

    Consider the natural $\Z \frac{1}2$-grading of $F_*\omega_{R_a}$ given by $(F_*\omega_{R_a})_{\frac{n}{2}} = F_*((\omega_{R_a})_n) = F_*((C_a)_{3n-1})$. After lifting $\Phi_{C_a}$ to $\Phi_{B_a}(F_*f_a-)$ in $\Hom_{B_a}(F_*B_a,B_a)$ by degree considerations, 
    \[
    \Phi_{C_a}\bigl(F_*((C_a)_{3n-1})\bigr)
    \subseteq
    \begin{cases}
        (C_a)_{\frac{3n}{2}-1}, & \text{if $n$ is even},\\
        0, & \text{if $n$ is odd}.
    \end{cases}
    \]
    Hence there is a well-defined restriction of the Frobenius dual map $\Phi_{C_a}:F_*C_a\to C_a$ to a map $\Phi_{R_a}:=\Phi_{C_a}:F_*\omega_{R_a}\to\omega_{R_a}$. To identify this restricted map with the canonical Frobenius dual on
    $R_a$, let $\mathfrak m\in\operatorname{MaxSpec}(A_a)$ and put $\mathfrak n_C=\mathfrak m C_a+(C_a)_+$ and $
    \mathfrak n_R=\mathfrak m R_a+(R_a)_+.$ The Veronese decomposition and the \v{C}ech complexes give a natural
    isomorphism
    \[
    H^9_{\mathfrak n_R}(R_a)
    \cong
    \left(H^9_{\mathfrak n_C}(C_a)\right)^{(3)}.
    \]
    This isomorphism commutes with Frobenius. Indeed, Frobenius sends
    relative degree $3d$ to relative degree $6d$. By local duality, the
    Matlis duals of these modules are the completions of the localized canonical modules
    $\omega_{R_a}$ and $\omega_{C_a}$, respectively, and the dual of
    Frobenius is the canonical Frobenius dual. Since $\Phi_{C_a}$ is the
    canonical Frobenius dual on $C_a$, its restriction to
    $F_*\omega_{R_a}$ is the canonical Frobenius dual on $R_a$. Under the compatible canonical-module identifications, functorial local duality identifies the completed localization of the restricted map with the completed canonical Frobenius dual. By faithful flatness of completion, the two maps agree after localization at $\mathfrak n_R$. Since $\omega_{R_a}$ is torsion-free over the domain $R_a$, the map
    \[
    \omega_{R_a}\longrightarrow(\omega_{R_a})_{\mathfrak n_R}
    \]
    is injective. Therefore the two globally defined maps $F_*\omega_{R_a}\to\omega_{R_a}$ agree.

    Let $m = x_1x_2x_3y\prod_{b\in\Lambda\setminus\{a\}}t_b\in B_a$. We note that $\Phi_{B_a}(F_*m) = 1$ and if $m'$ is a reduced monomial in $B_a$ properly dividing the monomial $m$, then $\Phi_{B_a}(F_*m')=0$. If $b\in \Lambda\setminus\{a\}$, then
    \[
    f_a\frac{m}{t_b} = m_a^2\frac{m}{t_b} + m_b^2m+ \sum_{c\in\Lambda\setminus\{a,b\}}t_cm_c^2\frac{m}{t_b}
    \]
    and
    \[
    \Phi_{B_a}\left(F_*f_a\frac{m}{t_b}\right) = m_b.
    \]
    Similarly,
    \[
    f_am = m_a^2m + \sum_{b\in\Lambda\setminus\{a\}} t_bm_b^2m
    \]
    and
    \[
    \Phi_{B_a}\left(F_*f_am\right)=m_a.
    \]

    The seven elements $\{m_b\mid b\in \Lambda\}$ generate $\omega_{R_a} = \oplus_{n\in\NN} (C_a)_{3n-1}$ as an $R_a$-module. The monomials $m$ and $\frac{m}{t_b}$ for $b\in \Lambda$ are multiples of $x_1x_2x_3\in (R_a)_1 = (C_a)_3$. Therefore $\Phi_{R_a}(F_*x_1x_2x_3-):F_*\omega_{R_a}\to \omega_{R_a}$ is surjective. The ring $(R_a)_{x_1x_2x_3}$ is strongly $F$-regular by the above. Hence there is $n\in\NN$ so that $(x_1x_2x_3)^n$ is a parameter test element of $R_a$. Let $e_0\in\NN$ be chosen so that $2^{e_0}-1\geq n$ and let $\Phi_{R_a}^{e_0}\in\Hom_{R_a}(F^{e_0}_*\omega_{R_a},\omega_{R_a})$ the $e_0$th iterate of the Frobenius dual map. Then $\Phi^{e_0}_{R_a}(F^{e_0}_*(x_1x_2x_3)^{2^{e_0}-1}-):F^{e_0}_*\omega_{R_a}\to \omega_{R_a}$ is surjective. Consequently, $R_a$ is $F$-rational, see \cite[Proposition~2.3]{DSPS25}.
\end{proof}

\begin{lemma}
    \label{lemma: Not F-injective lemma}
    Let $k=\overline{\mathbb F_2}$ and $A=k[t_1,\ldots,t_6]$. Let $q_1,\ldots,q_6$ be an ordering of $x_1^2, x_1x_2, x_1x_3, x_2^2, x_2x_3, x_3^2$, $ f=y^2+\sum_{i=1}^6t_iq_i^2$ a degree $4$ element of $A[x_1,x_2,x_3,y]$, $C = A[x_1,x_2,x_3,y]/(f)$, and $R=(C)^{(3)}$. Let $h\in A$ be a nonzero linear form in the six variables $t_1,\ldots,t_6$, $\fn = (t_1,t_2,\ldots,t_6,x_1,x_2,x_3,y)C$, and $\fm = \fn\cap R$. Then for every $\epsilon\in \fn^{[2]}\cap R$, $R/(h+\epsilon)$ is not $F$-injective.
\end{lemma}

\begin{proof}
    Identify $\omega_R = \oplus_{n\geq 1} C_{3n-1}$ as in the proof of \cref{lem: F rationality of Veronese general} and identify the Frobenius dual $\Phi_R\in\Hom_R(F_*\omega_R,\omega_R)$ as the Frobenius dual of $C$ with the restricted domain. We will first show $R_{\fm}/hR_{\fm}$ is not $F$-injective. Let $\Phi$ denote the standard Frobenius dual of $A[x_1,x_2,x_3,y]$ with respect to the $p$-basis $\{t_1,\ldots,t_6,x_1,x_2,x_3,y\}$. By \cref{proposition: Fedder for Frobenius dual}, the Frobenius dual of $R/(h)$ lifts to $\Phi_R(F_*h-):F_*\omega_R\to \omega_R$. If $g\in A[x_1,x_2,x_3,y]$ is the lift of an element of $\omega_R$, then $\Phi_R(F_*g)$ lifts to $\Phi(F_*fg) = y\Phi(F_*g) + \sum_{i=1}^6q_i\Phi(F_*t_ig)$. In particular, the lift of an element $\Phi_R(F_*hg)$ in $A[x_1,x_2,x_3,y]$ is of the form
    \begin{align*}
    \Phi_R(F_*hg) &= y\Phi(F_*hg) + \sum_{i=1}^6q_i\Phi(F_*t_ihg).
    \end{align*}
    By the grading, the coefficients on the right belong to $A[x_1,x_2,x_3,y]^{(3)}$, so their classes modulo $(f)$ belong to $R$.
    
    As vector space over $R/\fm$, $\omega_R/\fm\omega_R$ is seven-dimensional with basis $\{\overline{y},\overline{q_1},\ldots,\overline{q_6}\}$. If $\varphi: A[x_1,x_2,x_3,y]\to k$ is the evaluation at the origin map, then the image of $\Phi_R(F_*hg)$ in $\omega_R/\fm\omega_R$ is
    \[
    \overline{\Phi_R(F_*hg)} = \varphi(\Phi(F_*hg))\overline{y} +  \sum_{i=1}^6\varphi(\Phi(F_*t_ihg))\overline{q_i}.
    \]

    Assume that $h=\sum_{i=1}^6\delta_it_i$ and let $\lambda: \omega_R/\fm\omega_R\to k$ be the linear transformation defined by $\overline{y}\mapsto 0$ and $\overline{q_i}\mapsto \sqrt{\delta_i}$. Then
    \begin{align*}
    \lambda\left(\overline{\Phi_R(F_*hg)}\right) &= \sum_{i=1}^6 \varphi(\Phi(F_*t_ihg))\sqrt{\delta_i}\\
    &=\varphi\left(\Phi\left(F_*\sum_{i=1}^6\delta_it_ihg\right)\right)\\
    &=\varphi(\Phi(F_*h^2g))\\
    &=\varphi(h\Phi(F_*g))\\
    &=0.
    \end{align*}
    Therefore for every $g\in \omega_R$, $\overline{\Phi_R(F_*hg)}$ belongs to the kernel of $\lambda: \omega_R/\fm\omega_R\to k$ and hence $\Phi_R(F_*h-)$ is not surjective. Equivalently, $R/(h)$ is not $F$-injective.

    Let $\epsilon\in \fn^{[2]}\cap R$ and abuse notation and let $\epsilon\in \fn^{[2]}\cap (A[x_1,x_2,x_3,y])^{(3)}$ denote a lift of $\epsilon$ in $A[x_1,x_2,x_3,y]$. Since $\Phi_R(F_*(\epsilon\omega_R))\subseteq \fm \omega_R$, we can repeat the above analysis with $h+\epsilon$ to find that for any $g\in \omega_R$,
    \[
    \lambda(\overline{\Phi_R((h+\epsilon)g)})= \lambda(\overline{\Phi_R(hg)}) = 0,
    \]
    i.e., $\overline{\Phi_R((h+\epsilon)g)}\in \ker(\lambda)$ and $\Phi_R(F_*h+\epsilon-)$ cannot be surjective. Because $h$ is a linear form in $t_1,\ldots,t_6$, $h + \epsilon$ is regular on $R$, by \cref{proposition: Fedder for Frobenius dual}, $\Phi_R(F_*h+\epsilon-)$ is the lift of the Frobenius dual of $R/(h+\epsilon)$ and $R/(h+\epsilon)$ is not $F$-injective.
\end{proof}

\subsection{Projective Compactification}\label{section Projective model}

We continue using the notation of \cref{section affine charts} in addition to the following notation and assigned bigrading.
\begin{itemize}
    \item $\displaystyle \mathcal{A}=k[T_a\mid a\in\Lambda]$ with $T_a$ variables of degree $(1,0)$.
    \item $\displaystyle \mathcal{B} = \mathcal{A}[x_1,x_2,x_3,y,z]$ with $x_1,x_2,x_3$ variables of degree $(0,1)$, $y$ a variable of degree $(0,2)$, and $z$ a variable of degree $(1,3)$.
    \item $F = \sum_{a\in\Lambda}T_a m_a^2\in \mathcal{B}$, a homogeneous equation of degree $(1,4)$.
    \item $\displaystyle \mathcal{C} = \frac{\mathcal{B}}{(F)}$.
    \item $\displaystyle \mathcal{R} = \bigoplus_{n\in\NN} \mathcal{C}_{(3n,6n)}$.
\end{itemize}

\begin{lemma}
    \label{lem: script R standard graded F-rational}
    $\mathcal{R}$ is a standard graded $k$-algebra.
\end{lemma}

\begin{proof}
    Consider a monomial term $\prod_{a\in\Lambda}T^{\nu_a}x_1^{a_1}x_2^{a_2}x_3^{a_3}y^bz^d$ of degree $n \geq 2$. Then $\sum_{a\in\Lambda}\nu_a + d = 3n$ and $a_1+a_2+a_3 +2b + 3d = 6n$. If $d\geq 2$, then $\left(\prod_{a\in\Lambda}T^{\nu_a}_ax_1^{a_1}\right)x_2^{a_2}x_3^{a_3}y^bz^d$ is divisible by a monomial $T_az^2\in \left(\mathcal{R}\right)_1$. If $d = 1$ then $a_1+a_2+a_3+2b = 6n-3\geq 3$ and $a_1+a_2+a_3+2b$ is odd. Therefore $x_1^{a_1}x_2^{a_2}x_3^{a_3}y^b$ is divisible by a monomial in $m\in\mathcal{E}_3$. By degree considerations, the monomial $\prod_{a\in\Lambda}T^{\nu_a}x_1^{a_1}x_2^{a_2}x_3^{a_3}y^bz^d$  is divisible by a monomial $T_aT_bzm$ for some $a,b\in \Lambda$ and $m\in\mathcal{E}_3$. Lastly, if $d = 0$ then $x_1^{a_1}x_2^{a_2}x_3^{a_3}y^b$ is divisible by a monomial in $\mathcal{E}_6$ and $\prod_{a\in\Lambda}T^{\nu_a}x_1^{a_1}x_2^{a_2}x_3^{a_3}y^b$ is divisible by a monomial of the form $T_aT_bT_cm$ for some $a,b,c\in\Lambda$ and $m\in\mathcal{E}_6$. This shows every monomial element of $\left(\mathcal{R}\right)_n$ with $n\geq 2$ is divisible by a monomial term of $\left(\mathcal{R}\right)_1$. Therefore $\mathcal{R}$ is indeed standard graded over $k$.
\end{proof}

Let $\mathfrak{X} = \Proj_k(\mathcal{R})$ denote the corresponding projective variety. We next prove that $\mathfrak{X}$ is $F$-rational and there is a closed immersion $\mathfrak{X}\hookrightarrow\PP^{2547}_k$ so that the intersection of a general hyperplane $H$ of $\PP^{2547}_k$ with $\mathfrak{X}$ is not $F$-injective. This is accomplished by identifying $2548$ monomial elements $V$ of $\left(\mathcal{R}\right)_1$ so that the associated linear system is very ample. Before this, we collect some observations about the affine charts $\cO_\mathfrak{X}(D_+(W)) \cong \left(\mathcal{R}[W^{-1}]\right)_0 \cong \left(\mathcal{C}[W^{-1}]\right)_{(0,0)}$ where $W\in \left(\mathcal{R}\right)_1$ is a monomial.

\begin{remark}\label{Remark on charts of X}
    Let $W\in\left(\mathcal{R}\right)_1$ be a monomial. Then $W$ is divisible by $T_a$ for some $a\in\Lambda$. The affine section ring $\cO_\mathfrak{X}(D_+(W))$ is realized as a homogeneous localization of an algebra that adjoins fractions to $\mathcal{C}[T_a^{-1}]$. For every $b\in \Lambda$ set $t_b = \frac{T_b}{T_a}$. Then $F = T_a\left(m_a^2 +\sum_{b\in\Lambda\setminus\{a\}}t_bm_b^2\right)$, $T_a$ is a unit of $\mathcal{C}[T_a^{-1}]$, and so 
    \[
    \mathcal{C}[T_a^{-1}]\cong C_a\left[z,T_a,T_a^{-1}\right]
    \]
    where $C_a$ is the hypersurface ring from \cref{section affine charts}. The bidegree of $\mathcal{C}$ persists through the isomorphism. For each $b\in \Lambda$, the bidegree of $t_b$ is $(0,0)$, $x_1,x_2,x_3$ have bidegree $(0,1)$, $y$ has bidegree $(0,2)$, and $z$ has bidegree $(1,3)$. Therefore for every monomial $W\in\left(\mathcal{R}\right)_1$, if $T_a$ divides $W$, we make an identification of the corresponding affine chart
    \[
    \left(\mathcal{R}[W^{-1}]\right)_0 \cong \left(\mathcal{C}[W^{-1}]\right)_{(0,0)} \cong \left(C_a\left[z,T_a,T_a^{-1}, W^{-1}\right]\right)_{(0,0)}.
    \]
\end{remark}

\begin{lemma}
    \label{lem: An affine cover of X}
    Let $\mathfrak{X} = \Proj_{k}(\mathcal{R})$ and
    \[
    V'=\{T_az^2\mid a\in\Lambda\}\cup \{T_a^3y^3\mid a\in\Lambda\}\cup\{T_a^3x_i^6\mid a\in\Lambda\mbox{ and }1\leq i\leq 3 \}.
    \]
    Then $V'$ is a collection of monomials of $\left(\mathcal{R}\right)_1$ whose associated linear system is base point free.
\end{lemma}

\begin{proof}
    It is clear each element of $V'$ belongs to $\left(\mathcal{R}\right)_1$. Let $n\geq 1$ and $W\in \left(\mathcal{R}\right)_n = \mathcal{C}_{(3n,6n)}$ a monomial. In addition to being divisible by some $T_a$, a monomial of degree $(3n,6n)$ is divisible by $z,y,$ or some $x_i$ with $1\leq i\leq 3$. Therefore a large enough power of a monomial of degree $(3n,6n)$ with $n\geq 1$ is divisible by a monomial of degree $(3,6)$ of the form $T_az^2, T_a^3y^3$ or $T_a^3x_i^6$ for some $a\in\Lambda$ and $1\leq i\leq 3$. It follows that the projective variety $X$ is covered by the affine open sets defined by the non-vanishing of the collection degree $(3,6)$ elements in $V'$. Equivalently, the associated linear system of $V'$ is base point free.
\end{proof}

\begin{corollary}
    \label{cor: X is F-rational}
    Let $\mathfrak{X} = \Proj_k(\mathcal{R})$.
    \begin{enumerate}[label=(\alph*)]
        \item For each closed point $x\in \mathfrak{X}$, the embedding dimension $\dim_k(\fm_x/\fm_{x}^2)\leq 20$.
        \item The projective variety $\mathfrak{X}$ is $F$-rational.
    \end{enumerate}
\end{corollary}

\begin{proof}
    Let $V'$ be the collection of monomials $(\mathcal{R})_{1}$ of \cref{lem: An affine cover of X} defining a base point free linear system. We will show that for each $W\in V'$, the associated affine chart $\cO_{\mathfrak{X}}(D_+(W))$ is a finitely generated $k$-algebra generated by at most $20$ elements and $\cO_{\mathfrak{X}}(D_+(W))$ is $F$-rational.

    Suppose first that $W = T_az^2$ so that by Remark~\ref{Remark on charts of X}
    \[
    \cO_\mathfrak{X}(D_+(T_az^2))
    \cong
    \left(C_a\left[z,T_a,T_a^{-1},(T_az^2)^{-1}\right]\right)_{(0,0)} = \left(C_a\left[z,z^{-1},T_a,T_a^{-1}\right]\right)_{(0,0)},
    \]
    Let $u$ be a homogeneous element of $C_a$ of degree $(0,\alpha)$ and suppose $uz^{\beta_1}T_a^{\beta_2}$ with $\beta_1,\beta_2\in \Z$ has bidegree $(0,0)$. Then $\beta_1+\beta_2 = 0$ and $\alpha + 3\beta_1 = 0$. Therefore there exists $n\geq 0$ so that $\alpha = 3n$ and $uz^{\beta_1}T_a^{\beta_2}$ is of the form $u\left(\frac{T_a}{z}\right)^{n}$ and $u\in (C_a)_{(0,3n)}$. 
    
    Let $n\geq 1$ and consider a monomial of bidegree $(0,0)$ takes the form
    \[
    x_1^{a_1}x_2^{a_2}x_3^{a_3}y^b\left(\frac{T_a}{z}\right)^c, \quad a_1+a_2+a_3+2b -3c = 0.
    \]
    Given such a term, $x_1^{a_1}x_2^{a_2}x_3^{a_3}y^b$ has weighted degree divisible by three. Factor out powers of $y^3\left(\frac{T_a}{z}\right)^2$. If the remaining $y$-exponent is $1$, factor out $yx_i\frac{T_a}{z}$ for some $1\leq i\leq 3$; if a remaining $y$-exponent is $2$, factor out $(yx_i)(yx_j)\left(\frac{T_a}{z}\right)^2$, allowing the possibility that $i=j$. The remaining $x$-monomial has degree divisible by three, so it factors into cubic monomials. Therefore as an $A_a$-algebra, $\cO_\mathfrak{X}(D_+(T_az^2))$ is generated by the monomials
    \[
    \left\{x_1^{a_1}x_2^{a_2}x_3^{a_3}\frac{T_a}{z}\mid a_1+a_2+a_3 = 3\right\} \bigcup \left\{y^3\left(\frac{T_a}{z}\right)^2\right\}\bigcup\left\{yx_i\frac{T_a}{z}\mid 1\leq i\leq 3\right\}.
    \]
    Combined with the six variables generating $A_a$ over $k$, $\cO_\mathfrak{X}(D_+(T_az^2))$ is a $k$-algebra generated by $6+10+1+3=20$ elements.
    
    Let $u$ be a bidegree homogeneous element of $C_a$ of degree $(0,\alpha)$ and suppose $uz^{\beta_1}T_a^{\beta_2}$ with $\beta_1,\beta_2\in \Z$ has bidegree $(0,0)$. Then $uz^{\beta_1}T_a^{\beta_2}$ is of the form $u\left(\frac{T_a}{z}\right)^{n}$ and $u\in (C_a)_{(0,3n)}$. It follows that $\cO_\mathfrak{X}(D_+(T_az^2)) \cong R_a$ with the isomorphism identifying a degree $(0,0)$ element $u\left(\frac{T_a}{z}\right)^{n}$ with $u\in R_a$. By \cref{lem: F rationality of Veronese general}\ref{lem: F rationality of Veronese general F-rational}, $\cO_\mathfrak{X}(D_+(T_az^2))$ is $F$-rational.

    Next consider the case $W = T_a^3y^3$ for some $a\in\Lambda$. By \cref{Remark on charts of X} we identify
    \[
    \cO_{\mathfrak{X}}(D_+(T_a^3y^3)) \cong \left(C_a\left[z,T_a,T_a^{-1}, (T_ay^3)^{-1}\right]\right)_{(0,0)} = \left(C_a\left[z,T_a,T_a^{-1}, y^{-1}\right]\right)_{(0,0)}.
    \]
    The algebras $(C_a)_{y}$ and $\left(C_a\left[z,T_a,T_a^{-1}, y^{-1}\right]\right)$ are regular, the latter of which admits $\left(C_a\left[z,T_a,T_a^{-1}, y^{-1}\right]\right)_{(0,0)}$ as a direct summand. Hence $\left(C_a\left[z,T_a,T_a^{-1}, y^{-1}\right]\right)_{(0,0)}$ is strongly $F$-regular and therefore $F$-rational.

    We next bound the number of generators of $\left(C_a\left[z,T_a,T_a^{-1}, y^{-1}\right]\right)_{(0,0)}$ as a $k$-algebra. Consider a monomial of bidegree $(0,0)$ that is of the form
    \[
    x_1^{\alpha_1}x_2^{\alpha_2}x_3^{\alpha_3}y^{\beta}z^{\gamma_1}T_a^{\gamma_2}
    \]
    with $\alpha_1,\alpha_2,\alpha_3,\gamma_1\geq 0$ and $\beta,\gamma_2\in\Z$. Then $\gamma_1 +\gamma_2 = 0$. Hence the monomial is of the form
    \[
    x_1^{\alpha_1}x_2^{\alpha_2}x_3^{\alpha_3}y^{\beta}\left(\frac{z}{T_a}\right)^{\gamma}
    \]
    with $\alpha_1,\alpha_2,\alpha_3,\gamma\geq 0, \beta\in\Z$, and $\alpha_1 + \alpha_2 + \alpha_3 + 3\gamma +2\beta = 0$. 

    This implies the total number of factors of $x_1,x_2,x_3,\frac{z}{T_a}$ is even. We can therefore group two at a time, with corresponding power of $y^{-1}$, and identify relative positive degree monomial generators over $A_a$ as the set
    \[
    \left\{\frac{x_ix_j}{y}\mid 1\leq i\leq j\leq3\right\}\bigcup \left\{\frac{zx_i}{T_ay^2}\mid 1\leq i\leq 3\right\} \bigcup \left\{\frac{z^2}{T_a^2y^3}\right\}.
    \]
 Therefore $\left(C_a\left[z,T_a,T_a^{-1}, y^{-1}\right]\right)_{(0,0)}$ is a $k$-algebra generated by at most $6+6+3+1=16$ elements.

    Finally consider the case $W = T_a^3x_i^6$ for some $a\in\Lambda$ and $1\leq i\leq 3$. Then
    \[
    \cO_\mathfrak{X}(D_+(T_a^3x_i^6)) \cong \left(C_a\left[z, T_a,T_a^{-1},x_i^{-6}\right]\right)_{(0,0)} = \left(C_a\left[z, T_a,T_a^{-1},x_i^{-1}\right]\right)_{(0,0)}.
    \]
    Similar to the previous case, $C_a[z, T_a,T_a^{-1},x_i^{-1}]$ is strongly $F$-regular, $\left(C_a\left[z, T_a,T_a^{-1},x_i^{-1}\right]\right)_{(0,0)}$ is a direct summand of $C_a[z, T_a,T_a^{-1},x_i^{-1}]$, and hence $\cO_\mathfrak{X}(D_+(T_a^3x_i^6))$ is $F$-rational.

     We next bound the number of generators of $\left(C_a\left[z, T_a,T_a^{-1},x_i^{-1}\right]\right)_{(0,0)}$ as a $k$-algebra. Consider a monomial of bidegree $(0,0)$ of the form $x_1^{\alpha_1}x_2^{\alpha_2}x_3^{\alpha_3}y^{\beta}z^{\gamma_1}T_a^{\gamma_2}\left(\frac{1}{x_i}\right)^{\delta}$ belonging to the chart. Then $\gamma_1+\gamma_2 = 0$, implying the monomial is of the form
    \[
    x_1^{\alpha_1}x_2^{\alpha_2}x_3^{\alpha_3}y^{\beta}\left(\frac{z}{T_a}\right)^{\gamma}\left(\frac{1}{x_i}\right)^{\delta}
    \]
    with $\alpha_1,\alpha_2,\alpha_3,\beta,\gamma,\delta\geq 0$ and $\alpha_1+\alpha_2+\alpha_3+2\beta +3\gamma - \delta = 0$.
    It follows that over $A_a$, $\left(C_a\left[z, T_a,T_a^{-1},x_i^{-1}\right]\right)_{(0,0)}.$ is generated by the set of monomials
    \[
    \left\{\frac{x_j}{x_i}\mid j\not=i\right\}\bigcup\left\{\frac{y}{x_i^2}, \frac{z}{T_ax_i^3}\right\}.
    \]
    Therefore as a $k$-algebra, $\left(C_a\left[z, T_a,T_a^{-1},x_i^{-1}\right]\right)_{(0,0)}$ is generated by $6+ 2+2 = 10$ elements.
\end{proof}

\begin{lemma}
    \label{lem: the embedding of the projective variety}
    For each $n\in\NN$, let $\mathcal{E}_n = \{x_1^ax_2^bx_3^cy^d\mid a+b+c+2d = n\}$. Let $\mathfrak{X} = \Proj_k(\mathcal{R})$ and consider the sets
\begin{align*}
    V_1 = \{T_az^2\mid a\in\Lambda\},\quad V_2 = \{T_a^2zm\mid a\in\Lambda\mbox{ and }m \in\mathcal{E}_3\}, \quad V_3 = \{T_a^2T_bm\mid a,b\in\Lambda\mbox{ and }m \in\mathcal{E}_6\}.
\end{align*}
    Then $V := V_1\cup V_2\cup V_3$ is a collection of $2548$ monomials in $\left(\mathcal{R}\right)_1$ that define a very ample linear system.
\end{lemma}

\begin{proof}
    It is clear that each element of $V$ is a monomial of $\left(\mathcal{R}\right)_1$, that $|V_1|=7$, $|V_2| = 7\times13=91$, $|V_3| = 49\times50 = 2450$, and $|V| = 7+91+2450 = 2548$. Let $V'$ be the collection of monomials described in \cref{lem: An affine cover of X}. Then $V'\subseteq V$ and hence $V$ defines a base point free linear system. 

    Let $\varphi_V:\mathfrak{X}\to \PP^{|V|-1}_k = \PP^{2547}_k$ be the associated morphism of $V$. It remains to check $\varphi_V$ is a closed immersion; to do so, it suffices to show that for each $W\in V'$, $\mathcal{O}_\mathfrak{X}(D_+(W)) = k\left[\frac{U}{W}\mid U\in V\right]$. Indeed, for each $W\in V'$ let $Z_W$ be the corresponding coordinate of $\PP^{2547}_k$. Assuming that for each $W\in V'$ that $\mathcal{O}_\mathfrak{X}(D_+(W))$ is generated by the sections $\frac{U}{W}$ as $U$ varies in $V$ implies $D_+(W)\subseteq D_+(Z_W)$ is a closed immersion. Hence $\varphi_V$ defines a closed immersion in the open set $\bigcup_{W\in V'}D_+(Z_W)\subseteq \PP^{2547}_k$. But $\varphi_V:\mathfrak{X}\to \PP^{2547}_k$ is proper, hence its image is closed, so $\varphi_V:\mathfrak{X}\hookrightarrow \PP^{2547}_k$ is a closed immersion as needed.

    It remains to check for every $W\in V'$ the sections $\{\frac{U}{W}\mid U\in V\}$ generate $\cO_{\mathfrak{X}}(D_+(W))$ as a $k$-algebra. First suppose $W = T_az^2$ for some $a\in\Lambda$. In the proof of \cref{cor: X is F-rational}, we identified the generators of $\cO_{\mathfrak{X}}(D_+(W))$ as an $A_a$-algebra to be 
    \[
    \left\{x_1^{a_1}x_2^{a_2}x_3^{a_3}\frac{T_a}{z}\mid a_1+a_2+a_3 = 3\right\} \bigcup \left\{y^3\left(\frac{T_a}{z}\right)^2\right\}\bigcup\left\{yx_i\frac{T_a}{z}\mid 1\leq i\leq 3\right\}.
    \]
    
    If $b\in\Lambda$ then
    \[
    t_b = \frac{T_bz^2}{T_az^2}\quad \mbox{and}\quad T_bz^2\in V_1.
    \]
    If $a_1+a_2+a_3 +2b = 3$, then 
    \[
    x_1^{a_1}x_2^{a_2}x_3^{a_3}y^b\frac{T_a}{z} = \frac{T_a^2zx_1^{a_1}x_2^{a_2}x_3^{a_3}y^b}{T_az^2} \quad \mbox{and}\quad T_a^2zx_1^{a_1}x_2^{a_2}x_3^{a_3}y^b\in V_2.
    \]
    The final generator is
    \[
    y^3\left(\frac{T_a}{z}\right)^2=\frac{T_a^3y^3}{T_az^2}\quad\mbox{and}\quad T_a^3y^3\in V_3,
    \]
    so
    \[
    \cO_\mathfrak{X}(D_+(T_az^2)) = k\left[\frac{U}{T_az^2}\mid U\in V\right]
    \]
    as needed.

    We now show that if $a\in \Lambda$, then $\cO_{\mathfrak{X}}(D_+(T_a^3y^3))$ is a $k$-algebra generated by the sections $\left\{\frac{U}{T_a^3y^3}\mid U\in V\right\}$. We again use the $A_a$-algebra generators of $\cO_{\mathfrak{X}}(D_+(T_a^3y^3))$ from the proof of \cref{cor: X is F-rational}, 
    \[
    \left\{\frac{x_ix_j}{y}\mid 1\leq i\leq j\leq3\right\}\bigcup \left\{\frac{zx_i}{T_ay^2}\mid 1\leq i\leq 3\right\} \bigcup \left\{\frac{z^2}{T_a^2y^3}\right\}.
    \]
    If $b\in\Lambda$, then
    \[
    t_b = \frac{T_a^2T_by^3}{T_a^3y^3}\quad\mbox{and}\quad T_a^2T_by^3\in V_3.
    \]
    If $1\leq i\leq j\leq 3$ then
    \[
    \frac{x_ix_j}{y} = \frac{T_a^3x_ix_jy^2}{T_a^3y^3}\quad\mbox{and}\quad T_a^3x_ix_jy^2\in V_3.
    \]
    If $1\leq i\leq 3$ then
    \[
    \frac{zx_i}{T_ay^2} = \frac{T_a^2zyx_i}{T_a^3y^3}\quad \mbox{and}\quad T_a^2zyx_i\in V_2.
    \]
    The final generator
    \[
    \frac{z^2}{T_a^2y^3} = \frac{T_az^2}{T_a^3y^3}\quad\mbox{and}\quad T_az^2\in V_1.
    \]
    This completes the argument that
    \[
    \cO_\mathfrak{X}(D_+(T_a^3y^3)) = k\left[\frac{U}{T_a^3y^3}\mid U\in V\right]
    \]

    The final class of charts we need to check is a chart of the form 
    \[
    \cO_\mathfrak{X}(D_+(T_a^3x_i^6)) \cong \left(C_a\left[z, T_a,T_a^{-1},x_i^{-6}\right]\right)_{(0,0)} = \left(C_a\left[z, T_a,T_a^{-1},x_i^{-1}\right]\right)_{(0,0)}.
    \]
    Again by the proof of \cref{cor: X is F-rational}, we have $A_a$-algebra generators of $\cO_\mathfrak{X}(D_+(T_a^3x_i^6))$ given by
    \[
    \left\{\frac{x_j}{x_i}\mid j\not=i\right\}\bigcup\left\{\frac{y}{x_i^2}, \frac{z}{T_ax_i^3}\right\}.
    \]
    If $b\in\Lambda$ then
    \[
    t_b = \frac{T_a^2T_bx_i^6}{T_a^3x_i^6}\quad\mbox{and}\quad T_a^2T_bx_i^6\in V_3.
    \]
    If $j\not=i$,
    \[
    \frac{x_j}{x_i} = \frac{T_a^3x_i^5x_j}{T_a^3x_i^6}\quad\mbox{and}\quad T_a^3x_i^5x_j\in V_3.
    \]
    As for the final two generators
    \[
    \frac{y}{x_i^2} = \frac{T_a^3x_i^4y}{T_a^3x_i^6}\quad\mbox{and}\quad T_a^3x_i^4y\in V_3
    \]
    and
    \[
    \frac{z}{T_ax_i^3} = \frac{T_a^2zx_i^3}{T_a^3x_i^6}\quad\mbox{and}\quad T_a^2zx_i^3\in V_2.
    \]
    Therefore $\cO_\mathfrak{X}(D_+(T_a^3x_i^6)) = k\left[\frac{U}{T_a^3x_i^6}\mid U\in V\right]$ as needed to complete the proof.  
\end{proof}

\begin{theorem}
\label{theorem: the embedding, F-rationality, failure of Bertini}
Let $\mathfrak{X}=\Proj_k(\mathcal R)$ and identify $\mathfrak{X}$ with its image
under the closed immersion into $\PP_k^{2547}$ associated with the closed embedding of the very ample linear system associated to $V$ of \cref{lem: the embedding of the projective variety}. There exists a dense open subset $U\subseteq(\PP_k^{2547})^\vee$ such that, for every $H\in U(k)$, the hyperplane section $\mathfrak{X}\cap H$ is not $F$-injective.
\end{theorem}
\begin{proof}
    Put $a_0=(0,0,0,1)$ and
    $\Lambda'=\Lambda\setminus\{a_0\}$. On the affine open set
    $D_+(T_{a_0}z^2)$, if we identify $t_a = \frac{T_a}{T_{a_0}}$ then 
    \[
        \cO_\mathfrak{X}(D_+(T_{a_0}z^2))
        \cong R_{a_0}=C_{a_0}^{(3)}
    \]
    is the affine ring of \cref{section quasi-projective example} and $C_{a_0} = \frac{A_{a_0}[x_1,x_2,x_3,y]}{(y^2 + \sum_{a\in\Lambda'}t_am_a^2)}$.

    Under the chart identification, a general hyperplane $H$ restricts to the equation $L=0$, where
    \[
        L=\sum_{a\in\Lambda}\delta_a t_a + \sum_{\substack{a\in\Lambda\\m\in\mathcal{E}_3}} \varepsilon_{a,m}t_a^2m+\sum_{\substack{a,b\in\Lambda\\m\in\mathcal{E}_6}} \gamma_{a,b,m}t_a^2t_bm.
    \]
    To simplify notation, label the following four elements of $\mathcal{E}_3$: $m_0 = x_1x_2x_3, m_1 = x_1y, m_2 = x_2y$, and $m_3 = x_3y$. Also denote
\begin{align*}
    h := \sum_{a\in\Lambda}\delta_a t_a,\quad \ell_i:=\sum_{a\in\Lambda} \sqrt{\varepsilon_{a,m_i}}t_a\text{ for }0\leq i\leq 3.
\end{align*}

    Choose a labeling of the six elements $a_1,a_2,a_3,a_4,a_5,a_6$ of $\Lambda'$. For each hyperplane $H$ and associated restriction $L=0$, let $M_H,N_H$ be the $5\times6$ matrices
    \[
    M_H:=\begin{pmatrix}
        \delta_{a_1} & \delta_{a_2} & \cdots & \delta_{a_6} \\
        \sqrt{\varepsilon_{a_1,m_0}} & \sqrt{\varepsilon_{a_2,m_0}} & \ldots & \sqrt{\varepsilon_{a_6,m_0}}\\
        \vdots & \vdots & & \vdots\\
        \sqrt{\varepsilon_{a_1,m_3}} & \sqrt{\varepsilon_{a_2,m_3}} & \cdots & \sqrt{\varepsilon_{a_6,m_3}}
    \end{pmatrix} \quad \mbox{and} \quad
    N_H:=\begin{pmatrix}
        \delta_{a_1}^2 & \delta_{a_2}^2 & \cdots & \delta_{a_6}^2 \\
        {\varepsilon_{a_1,m_0}} & {\varepsilon_{a_2,m_0}} & \ldots & {\varepsilon_{a_6,m_0}}\\
        \vdots & \vdots & & \vdots\\
        {\varepsilon_{a_1,m_3}} & {\varepsilon_{a_2,m_3}} & \cdots & {\varepsilon_{a_6,m_3}}
    \end{pmatrix}.
    \]
    Let $U\subseteq (\PP^{2547}_k)^{\vee}$ be the open subset consisting of hyperplanes $H$ so that the associated $5\times6$ matrix $N_H$ has maximal rank $5$. Equivalently, if $H\in U(k)$, then the linear parts of $h,\ell_0,\ell_1,\ell_2,\ell_3$ are linearly independent. Writing $\vec{t}=(t_{a_1},\ldots,t_{a_6})^{\mathsf T}$ and
    letting $\vec c_H\in k^5$ be the vector of constant terms,
    the equations defining the vanishing locus $V(h,\ell_0,\ell_1,\ell_2,\ell_3)$ are
    \[
        M_H\vec{t}=-\vec c_H.
    \]
    The linear map $M_H:k^6\to k^5$ is surjective, so this
    system has a solution $\vec{\tau}\in k^6$.
    Moreover, $\dim_k\ker(M_H)=1$. Consequently,
    \[
        \Gamma_H
        :=V(h,\ell_0,\ell_1,\ell_2,\ell_3)
        =\vec\tau+\ker(M_H)
        \cong\mathbb A_k^1
    \]
    is a nonempty affine line in $\Spec(A_{a_0})\cong\mathbb A_k^6$. Identify $\Gamma_H$ with its image in $V((R_{a_0})_+)\subseteq \mathfrak{X}$, then $L-h\in(R_{a_0})_+$ and $h$ vanishes on $\Gamma_H$. Hence the restriction of $L$ to $\Gamma_H$ is zero. Thus $\Gamma_H\subseteq \mathfrak{X}\cap H$. 
    
    Let $\mathbf{\tau} = (\tau_{a_1},\tau_{a_2},\ldots,\tau_{a_6})\in \Gamma_H$. Consider the change of variables defined by $\overline{t_{a_i}} = t_{a_i}+\tau_{a_i}$ and $\overline{y} = y+\sum_{i=1}^6\sqrt{\tau_{a_i}}m_{a_i}$. Under the change of variables, the hypersurface equation defining $C_{a_0}$ becomes
    \[
    y^2 + \sum_{i=1}^6t_{a_i}m_{a_i}^2 = y^2 + \sum_{i=1}^6\overline{t_{a_i}}m_{a_i}^2 + \sum_{i=1}^6\tau_{a_i}m_{a_i}^2 = \overline{y}^2 + \sum_{i=1}^6\overline{t_{a_i}}m_{a_i}^2.
    \]
    The point $\mathbf{\tau}\in \Gamma_H$ is chosen so that $h(\mathbf{\tau}) = \ell_0(\mathbf{\tau}) = \ell_1(\mathbf{\tau})=\ell_2(\mathbf{\tau})=\ell_3(\mathbf{\tau}) = 0$. Therefore under the change of coordinates, the functions $h,\ell_0,\ell_1,\ell_2,\ell_3$ have $0$ constant term. Explicitly using $t_{a_0}=1$,
    \[
    h = \sum_{i=0}^6\delta_{a_i}t_{a_i} =\delta_{a_0} +  \sum_{i=1}^6\delta_{a_i}\overline{t_{a_i}} + \sum_{i=1}^6\delta_{a_i}\tau_{a_i} = \sum_{i=1}^6\delta_{a_i}\overline{t_{a_i}}
    \]
    and for each $0\leq j\leq 3$,
    \[
    \ell_{j} = \sum_{i=0}^6 \sqrt{\varepsilon_{a_i,m_j}}\,t_{a_i} = \sum_{i=1}^6 \sqrt{\varepsilon_{a_i,m_j}}\,\overline{t_{a_i}}.
    \]

    Consider the translation of $L$ under the change of variables. The initial sum $\sum_{a\in\Lambda}\delta_at_a$ defining $L$ is $h$. The middle sum $\sum_{\substack{a\in\Lambda\\m\in\mathcal{E}_3}} \varepsilon_{a,m}t_a^2m$ expands to \[ \sum_{\substack{a\in\Lambda\\m\in\mathcal{E}_3}} \varepsilon_{a,m}t_a^2m = \ell_0^2m_0 + \ell_1^2m_1 +\ell_2^2m_2 + \ell_3^2m_3 + \sum_{\substack{a\in\Lambda \\ m\in\mathcal{E}_3\setminus\{m_0,m_1,m_2,m_3\}}}\varepsilon_{a,m}t_a^2m. \] The monomials of $\mathcal{E}_3\setminus\{m_0,m_1,m_2,m_3\}$ are the nine monomials of degree three in $x_1,x_2,x_3$ other than $x_1x_2x_3$. The last sum defining $L$, namely $\sum_{\substack{a,b\in\Lambda\\m\in\mathcal{E}_6}} \gamma_{a,b,m}t_a^2t_bm$ is a polynomial over $A_{a_0}$ in the variables $x_1,x_2,x_3,y$ of degree $6$. Therefore the change of variables gives $L$ the form
    \[
    L = h + \sum_{i=0}^3\ell_i^2g_i + w + P
    \]
    where
    \begin{itemize}
        \item $h$ is a nonzero linear form in the $6$ variables $\overline{t_{a_1}},\ldots,\overline{t_{a_6}}$;
        \item $g_i\in A_{a_0}[x_1,x_2,x_3,y]$ is homogeneous of degree three;
        \item $w$ is an $A_{a_0}$-linear combination of cubic monomials in $x_1,x_2,x_3$ other than $x_1x_2x_3$;
        \item $P$ is homogeneous of degree $6$.
    \end{itemize}
    If $\fn$ is the maximal ideal of the origin of $C_{a_0}$, then $\sum_{i=0}^3\ell_i^2g_i+w+P\in \fn^{[2]}\cap R_{a_0}$. Thus, $L$ satisfies the hypotheses of \cref{lemma: Not F-injective lemma} and $R_{a_0}/(L)$ is not $F$-injective.
\end{proof}

\begin{theorem}
    \label{thm: the projective example embeds in P^{25}}
    Let $\mathfrak{X} = \Proj_k(\mathcal{R})$. There exists a closed immersion $\mathfrak{X}\hookrightarrow \PP^{29}_k$ and a dense open subset $U\subseteq \left(\PP^{29}_k\right)^{\vee}$ so that for every $H\in U(k)$, $\mathfrak{X}\cap H$ is not $F$-injective.
\end{theorem}

\begin{proof}
    Let $\varphi_V:\mathfrak{X}\hookrightarrow \PP^{2547}_k$ the closed embedding of the very ample linear system associated to $V$ of \cref{lem: the embedding of the projective variety}, and $U\subseteq \left(\PP^{2547}_k\right)^\vee$ the dense open subset of Theorem~\ref{theorem: the embedding, F-rationality, failure of Bertini} so that for each $H\in U$, $\mathfrak{X}\cap H$ is not $F$-injective. Since $\dim \mathfrak{X}=9$, the secant variety of $\mathfrak{X}$ has dimension at most $2\cdot 9+1=19$, see \cite[\S~1.1]{Adl87} Consequently, a general linear subspace $\mathcal{L}\cong\PP^{2527}_k\subseteq\PP^{2547}_k$ is disjoint from the secant variety. Projection from $\mathcal{L}$ therefore induces a morphism $\mathfrak{X}\to\PP^{19}_k$ that separates distinct closed points.

    For each $x\in \mathfrak{X}$ let $\mathbb{T}_x\mathfrak{X}$ be the embedded projective tangent space of $x\in \mathfrak{X}$. A fiber of $$\{(x,p)\in \mathfrak{X}\times\PP^N_k\mid p\in\mathbb{T}_x\mathfrak{X}\}\to \mathfrak{X}$$ has dimension at most $20$ by \cref{cor: X is F-rational}. By \cite[\href{https://stacks.math.columbia.edu/tag/02JS}{Tag 02JS}]{stacks-project}, 
    \[
    \dim\overline{\bigcup_{x\in \mathfrak{X}}\mathbb{T}_x\mathfrak{X}}\leq \dim(\mathfrak{X}) + 20 = 9+ 20 = 29.
    \]
    Therefore a general linear subspace $\mathcal{L}\cong \PP^{2547-30}_k$ is disjoint from $\overline{\bigcup_{x\in \mathfrak{X}}\mathbb{T}_x\mathfrak{X}}$ and projection from such an $\mathcal{L}$ induces a morphism $\mathfrak{X}\to \PP^{29}_k$ that separates tangent directions. 
    
    Let $\mathcal{L}\cong \PP^{2547-30}_k$ be a general linear subspace of $\PP^{2547}_k$. Projection from $\mathcal{L}$ induces a morphism $\psi:\mathfrak{X}\to \PP^{29}_k$ whose associated linear system separates points and tangent vectors, i.e., $\psi$ is a closed immersion. The linear subspace $\mathcal{L}^\perp\cong(\PP^{29}_k)^\vee\subseteq(\PP^{2547}_k)^\vee$ parameterizing hyperplanes containing $\mathcal{L}$ meets $U$. Thus $U\cap\mathcal{L}^\perp$ is a nonempty dense open subset of $\mathcal{L}^\perp$. The corresponding hyperplane sections under $\psi$ agree scheme-theoretically with those under $\varphi_V$. Consequently, a general hyperplane section of $\psi(\mathfrak{X})$ is not $F$-injective.
\end{proof}

\printbibliography

\end{document}